\documentclass[11 pt]{amsart}
\usepackage{color}

\usepackage{color}
\usepackage[usenames,dvipsnames]{xcolor}
\usepackage{fullpage}
\usepackage{verbatim}
\usepackage{amssymb}
\usepackage{amscd}
\usepackage{amsmath}
\usepackage{eucal}
\usepackage{setspace}
\usepackage{url}
\usepackage[utf8]{inputenc}
\providecommand{\up}[1]{{\upshape(}#1{\upshape)}}
\newcommand{\tst}{\textstyle}

\newcommand{\fixmehide}[1]{}
\newcommand{\fixmelater}[1]{}
\newcommand{\fixmedone}[1]{}
\newcommand{\fixmehidden}[1]{}
\newcommand{\tmop}[1]{\ensuremath{\operatorname{#1}}}

\def\sumstar{\sideset{}{^*}\sum}
\def\sumfund{\sideset{}{^\flat}\sum}
\def\pamod{\! \! \! \! \pmod}

\allowdisplaybreaks
\def\eps{\varepsilon}
\newcommand{\Tr}{\mathrm{Tr}}
\newcommand{\AI}{{\mathcal{AI}}}
\newcommand{\new}{\mathrm{new}}
\newcommand{\std}{\mathrm{std}}
\newcommand{\ad}{\mathrm{ad}}
\newcommand{\Sym}{\mathrm{Sym}}
\renewcommand{\Im}{\mathrm{Im}}
\renewcommand{\Re}{\mathrm{Re}}
\newcommand{\Sp}{\mathrm{Sp}}
\renewcommand{\H}{\mathbb H}
\newcommand{\T}[1]{{}^t\!{#1}}

\newcommand{\A}{{\mathbb A}}
\renewcommand{\P}{{\mathcal P}}
\newcommand{\Q}{{\mathbb Q}}
\newcommand{\Z}{{\mathbb Z}}
\newcommand{\R}{{\mathbb R}}
\newcommand{\C}{{\mathbb C}}
\newcommand{\B}{{\mathcal B}}
\newcommand{\bs}{\backslash}

\newcommand{\GL}{{\rm GL}}

\newcommand{\SL}{{\rm SL}}
\newcommand{\SO}{{\rm SO}}

\newcommand{\GSp}{{\rm GSp}}
\newcommand{\PGSp}{{\rm PGSp}}

\newcommand{\St}{{\rm St}}

\newcommand{\disc}{{\rm disc}}

\DeclareMathOperator{\Cl}{Cl}
\newcommand{\mat}[4]{{\setlength{\arraycolsep}{0.5mm}\left(\begin{array}{cc}#1&#2\\#3&#4\end{array}\right)}}

\newcommand{\forget}[1]{}

\AtBeginDocument{%
   \def\MR#1{}
}

\newtheorem{lemma}{Lemma}[section]
\newtheorem{theorem}[lemma]{Theorem}
\newtheorem{corollary}[lemma]{Corollary}
\newtheorem{proposition}[lemma]{Proposition}

\theoremstyle{remark}
\newtheorem{remark}[lemma]{Remark}

\begin{document}

\bibliographystyle{amsalpha}

\title[On fundamental Fourier coefficients of Siegel cusp forms]{On fundamental Fourier coefficients of Siegel cusp forms of degree 2}

\author{Jesse J\"{a}\"{a}saari}
\address{School of Mathematical Sciences\\
  Queen Mary University of London\\
  London E1 4NS\\
  UK}
  \email{j.jaasaari@qmul.ac.uk}

\author{Stephen Lester}
\address{Department of Mathematics \\ King's College London \\ London W2CR 2LS \\
  UK}
  \email{steve.lester@kcl.ac.uk}

\author{Abhishek Saha}
\address{School of Mathematical Sciences\\
  Queen Mary University of London\\
  London E1 4NS\\
  UK}
  \email{abhishek.saha@qmul.ac.uk}


 \begin{abstract}Let $F$ be a Siegel cusp form of degree 2, even weight $k \ge 2$ and odd squarefree level $N$. We undertake a detailed study of the analytic properties of Fourier coefficients $a(F,S)$ of $F$ at fundamental matrices $S$ (i.e., with $-4\det(S)$ equal to a fundamental discriminant). We prove that as $S$ varies along the equivalence classes of fundamental matrices with  $\det(S) \asymp X$, the sequence $a(F,S)$ has  at least $X^{1-\eps}$ sign changes, and takes at least $X^{1-\eps}$ ``large values". Furthermore, assuming the Generalized Riemann Hypothesis as well as the refined Gan--Gross--Prasad conjecture, we prove the bound $|a(F,S)| \ll_{F, \eps} \frac{\det(S)^{\frac{k}2 - \frac{1}{2}}}{ (\log |\det(S)|)^{\frac18 - \eps}}$ for fundamental matrices $S$.
 \end{abstract}

 \maketitle

\section{Introduction}
Let $S_k(\Gamma_0^{(2)}(N))$ denote the space of Siegel cusp forms of degree  $2$ and weight $k$ with respect to the congruence subgroup $\Gamma_0^{(2)}(N) \subseteq \Sp_4(\Z)$ of level $N$. Any $F \in S_k(\Gamma_0^{(2)}(N))$ has a Fourier expansion of the form
\begin{equation}\label{fouriercoeffsiegelform}\tst
 F(Z)=\sum_{S\in \Lambda_2} a(F, S)e^{2\pi i{\rm Tr}(SZ)},\qquad Z\text{ in the Siegel upper half space},
\end{equation}
where the set $\Lambda_2$ consists of symmetric, semi-integral, positive-definite $2\times2$ matrices $S$, i.e.,
\begin{equation}\label{e:Lambda2} \tst
 \Lambda_2 = \left\{\mat{a}{b/2}{b/2}{c},\qquad a,b,c\in\Z, \qquad a>0, \qquad  d:=b^2 - 4ac <0\right\}.
\end{equation}
For $S =\mat{a}{b/2}{b/2}{c}\in \Lambda_2$, we define $\disc(S)=-4\det(S)=b^2-4ac.$ If $d = \disc(S)$ is a fundamental discriminant\footnote{Recall that an integer $n$ is a fundamental discriminant if \emph{either} $n$ is a squarefree integer congruent to 1 modulo 4 \emph{or} $n = 4m$ where $m$ is a squarefree integer congruent to 2 or 3 modulo 4.}, then $S$ is called \emph{fundamental}. A fundamental $S$ is automatically primitive (i.e., $\gcd(a,b,c)=1$). Observe that if $d$ is odd, then $S$ is fundamental if and only if $d$ is squarefree.
 The fundamental Fourier coefficients of Siegel cusp forms are deep and highly interesting objects. These are the basic building blocks, in the sense that one cannot use the theory of Hecke operators to relate the Fourier coefficients $a(F,S)$ at these matrices to those at simpler matrices. Furthermore, fundamental Fourier coefficients are closely related to central $L$-values.

 In \cite{saha-2013,sahaschmidt} it was proved that if $k>2$ is even and $N$ is squarefree, then elements of $S_k(\Gamma_0^{(2)}(N))$ (under some mild assumptions) are uniquely determined by their fundamental Fourier coefficients. More precisely, it was proved there that \emph{for $k$, $N$ as above, if $F \in S_k(\Gamma_0^{(2)}(N))$ is non-zero and an eigenfunction of the $U(p)$ operators for $p|N$, then $a(F, S) \ne 0$ for infinitely many matrices $S$ such that $\disc(S)$ is odd and squarefree.} This non-vanishing result is crucial for the existence of good Bessel models \cite[Lemma 5.1.1]{transfer} and consequently was needed for removing a key assumption from theorems due to Furusawa~\cite{fur}, Pitale--Schmidt~\cite{pitsch} and the third-named author~\cite{saha1, sah2, transfer} on the degree 8 $L$-function on $\GSp_4 \times \GL_2$. Furthermore, there is a remarkable identity, explained in more detail in Section \ref{s:lvaluesintro} below, relating squares of (weighted averages of) fundamental Fourier coefficients and central values of dihedral twists of $\GSp_4$ and $\GL_2$ $L$-functions. Indeed, the fundamental Fourier coefficients are \emph{unipotent periods} whose weighted averages are \emph{Bessel periods} whose absolute squares are essentially central $L$-values of degree 8 $L$-functions, via the refined Gan--Gross--Prasad conjectures \cite{yifengliu}.

 Motivated by these connections, the objective of this paper is to understand better the nature of the fundamental Fourier coefficients. In particular, we investigate the following questions:

 \begin{itemize}
\item Are there many \emph{sign changes} among the fundamental Fourier coefficients?
\item How \emph{large} (in the sense of both lower and upper bounds) are the fundamental Fourier coefficients?
\end{itemize}

We emphasize that while these kinds of questions have been previously studied for the full sequence $a(F,S)$ ($S\in \Lambda_2$) of Fourier coefficients attached to $F$ (see \cite{Choie-gun-kohnen-15,Gun-Sengupta-17,He-Zhao-18} for results on sign-changes and \cite{das2020large} for an Omega-result) there appears to be virtually no previous work in the more subtle setting where one restricts to fundamental Fourier coefficients.  There has also been a fair bit of work on sign changes of \emph{Hecke eigenvalues} of Siegel cusp forms \cite{kohnen07, pitschsign, RSW16, Kohnen-das-18} which can be combined with the Hecke relations \cite{andrianov} to deduce sign changes among the $a(F,S)$ with $\disc(S)=dm^2$ where $d$ is a \emph{fixed} fundamental discriminant and $m$ varies. This should make it clear that the problem of obtaining sign changes or growth asymptotics for Fourier coefficients \emph{not} associated to fundamental discriminants is of a different flavor (and
relatively easier). Our focus in this paper is on the subsequence of Fourier coefficients $a(F,S)$ with $S$ restricted to matrices of fundamental discriminant, where these questions are more difficult.

\subsection{Main results}\label{s:signintro} \label{s:intromainresults}
Let $k > 2$ be an even integer and $N$ be an odd squarefree integer. Fix $F \in S_k(\Gamma_0^{(2)}(N))$. If $N>1$, assume that $F$ is an eigenform for the $U(p)$ Hecke operator (see \eqref{upaction}) for the finitely many primes $p|N$; we make no assumptions concerning whether $F$ is a Hecke eigenform at primes not dividing the level $N$. Our main result on sign changes is as follows.
\medskip

\noindent
\textbf{Theorem A.} (see Theorem \ref{thm:signssiegel}) \emph{For $F$ as above with real Fourier coefficients one can fix $M$ such that given $\eps>0$ and  sufficiently large $X$, there  are $ \ge X^{1-\eps}$ distinct odd squarefree integers $n_i \in [X, MX]$  and associated fundamental matrices $S_i\in \Lambda_2$ with $|\disc(S_i)| =n_i$, such that with the $n_i$ ordered in increasing manner  we have  $a(F, S_{i})a(F, S_{i+1})<0$.}

\medskip
Thus, Theorem A asserts that there are at least $X^{1-\eps}$ (strict) sign changes among the fundamental Fourier coefficients of discriminant $\asymp X$. Interestingly, this also improves the exponent of the non-vanishing results of \cite{saha-2013, sahaschmidt} mentioned earlier, where it was proved that there are $\gg_\eps X^{5/8 - \eps}$ non-vanishing fundamental Fourier coefficients of discriminant up to $X$.

Another question left unanswered in all previous works is that of \emph{lower bounds} for $|a(F,S)|$ with $S$ fundamental. Let $F$ be as above and fixed. A famous (and very deep) conjecture of Resnikoff and Saldana \cite{res-sald} predicts that for $S$ a fundamental\footnote{The conjecture extends to non-fundamental matrices but then it needs to be modified slightly by excluding the Saito-Kurokawa lifts.} matrix in $\Lambda_2$, \begin{equation}\label{e:ressald}|a(F,S)| \ll_{F,\varepsilon} |\disc(S)|^{\frac{k}2 -\frac{3}4 + \varepsilon}. \end{equation}
We prove a lower bound for many fundamental Fourier coefficients with an exponent of the same strength.

\medskip

\noindent
\textbf{Theorem B.} (see Theorem \ref{thm:large}) \emph{For $F$ as above, $\eps>0$ fixed, and all sufficiently large $X$ there are $\ge X^{1-\eps}$ distinct odd squarefree  integers $n \in  [X, 2X]$, with associated fundamental matrices $S_n$ such that $|\disc(S_n)|=n$ and  \[ \tst |a(F,S_n)| \ge n^{\frac{k}2 - \frac34} \exp\left(\frac{1}{82} \sqrt{\frac{\log n}{\log \log n}} \right). \]
}

Theorem B tells us that there are at least $X^{1-\eps}$ fundamental Fourier coefficients of discriminant $\asymp X$ whose sizes are ``large". Incidentally, just like Theorem A, Theorem B also improves upon the  exponent  of the set of non-vanishing fundamental coefficients obtained in \cite{sahaschmidt} from $5/8$ to 1.

Next, we investigate \emph{upper bounds} for the Fourier coefficients $|a(F,S)|$ for fundamental $S$. The best currently known bound  is due to Kohnen \cite{Koh93} who proved that $|a(F,S)| \ll_{F, \eps} |\disc(S)|^{\frac{k}2 - \frac{13}{36}+\eps}.$ This bound is quite far from the conjectured true bound \eqref{e:ressald}. In fact, even if one were to assume the Generalized Lindel\"of hypothesis, one only obtains the upper bound  $\ll_{F, \eps}|\disc(S)|^{\frac{k}2 - \frac{1}{2}+\eps}$ (as explained further below). Thus, the exponent $\frac{k}2 - \frac{1}{2}$ appears to be a natural barrier. By employing probabilistic methods and assuming the Generalized Riemann Hypothesis (GRH) for several $L$-functions, we are able to go beyond this barrier for the first time.

\medskip

\noindent
\textbf{Theorem C.} (see Theorem \ref{thm:upperbd}) \emph{Let $k > 2$ be an even integer and $N$ be an odd squarefree integer. Fix $F \in S_k(\Gamma_0^{(2)}(N))$. Assume that the refined Gan--Gross--Prasad conjecture \cite[(1.1)]{yifengliu}  holds\footnote{A proof of this conjecture has recently been announced by Furusawa and Morimoto (RIMS conference ``Analytic, geometric and $p$-adic aspects of automorphic forms and $L$-functions", January 2020).} for Bessel periods of holomorphic cusp forms on $\SO_5(\A)$. Assume that GRH holds for $L$-functions in the Selberg\footnote{We assume GRH for Rankin-Selberg $L$-functions $L(s, \pi_1 \times \pi_2)$ and for symmetric square $L$-functions $L(s, \Sym^2 \pi_3)$ where for $i=1,2,3$, $\pi_i$ is an automorphic representation of $\GL_{n_i}(\A_K)$ with $n_i \le 5$ with $K$ equal to either $\Q$ or an imaginary quadratic field; see also Remark \ref{r:automorphy}.} class. Then we have \[ \tst |a(F,S)| \ll_{F, \eps} \frac{|\disc(S)|^{\frac{k}2 - \frac{1}{2}}}{ (\log |\disc(S)|)^{\frac18 - \eps}}\]
for  fundamental matrices $S$.
}

We note that a bound similar to that obtained in Theorem C has been recently proved in the special case where $F$ is a \emph{Yoshida lift} by Blomer and Brumley \cite[Corollary 4]{Blomer-Brumley}.

\subsection{The reduction of Theorems A and B to half-integral weight forms} The proofs of Theorems A and B rely on reducing these questions to corresponding ones about cusp forms of weight $k-\frac12$ on the upper half-plane, exploiting the Fourier-Jacobi expansion of $F$ and the relation
between Jacobi forms and classical cusp forms of half-integral weight. More precisely, using \cite{iwanprime} it follows (see Section \ref{s:constructhalf}) that the set of primes $p$ such that the $p$'th Fourier-Jacobi coefficient of $F$ is non-zero has positive density in the set of all primes; fix \emph{any} $p$ in this set coprime to $N$. Using a classical construction going back to Eichler and Zagier \cite[Thm 5.6]{eichzag} in the case $N=1$ and due to Manickam--Ramakrishnan \cite{manickram} for squarefree $N$, we can now construct a non-zero cusp form $h$ of level $4Np$ and weight $k-\frac12$ whose Fourier coefficients $a(h,n)$ essentially equal some  $a(F,S)$ with $|\disc(S)|=n$.

From the above construction, Theorem A will follow if we can demonstrate $X^{1-\eps}$ sign changes among the coefficients $a(h,n)$ of the half-integral weight form $h$ for odd squarefree $n\asymp X$, which is exactly what we prove in Theorem \ref{thm:signs}, a result which builds upon works of Matom\"aki and Radziwi{\l}{\l} \cite{MatomakiRadziwill} and \cite{LesterRadziwill2019signs} and may be of independent interest. A point worth noting here is that $h$ is \emph{not} typically a Hecke eigenform (even when $F$ is a Hecke eigenform) as the passage from Siegel cusp forms to Jacobi forms described above is not a functorial correspondence. 
The main ingredient for our proof of Theorem \ref{thm:signs} is the demonstration of cancellation in sums of
$a(h,n)$ over almost all short intervals together with bounds on their moments thereby providing a lower bound on sums of $|a(h,n)|$ over almost all short intervals. Combining the two results shows that over many short intervals the absolute value of the average of $a(h,n)$ is strictly smaller than that of $|a(h,n)|$. Consequently, a sign change of $a(h,n)$ occurs in many short intervals.

Likewise, Theorem B follows provided we can demonstrate suitable large values for $|a(h,n)|$. This is done in Section \ref{s:large}. The main result of that section, Theorem \ref{thm:large2}, says that there are at least $X^{1-\eps}$ odd squarefree $n\asymp X$ with $$\tst |a(h,n)| \ge n^{\frac{k}2 - \frac34} \exp\left(\frac{1}{82} \sqrt{\frac{\log n}{\log \log n}} \right).$$ Theorem \ref{thm:large2} generalizes recent work of Gun--Kohnen--Soundararajan \cite{gun2020large} which dealt with the case of $h$ of level 4. The proof of Theorem  \ref{thm:large2} follows the ``resonance method"-strategy of \cite{soundararajan08,gun2020large}; however there are additional complications coming from the level which we need to overcome. The starting point of the proof is to use Kohnen's basis for $S_{k+\frac12}^+(4N)$ consisting of newforms and an explicit form of Waldspurger's formula to reduce the problem to showing large values for (a weighted average of) a particular central $L$-value, while controlling sizes of certain other central $L$-values (see Proposition \ref{p:largep1} and the discussion after it, in particular estimates (\ref{estimate1}) and (\ref{estimate2})). This is achieved by the resonance-method as in \cite{gun2020large}. A key technical input for this method is the evaluation of the first moment of twisted central $L$-values (Proposition \ref{key-proposition}), which is obtained following the method of \cite{SoundYoung}. Complications arising from the level show up here in a form of extra congruence and coprimality conditions, and these are dealt with as in \cite{RadziwillSound}.

Theorem C, unlike Theorems A and B, does not involve a reduction to half-integral weight forms. We explain the main ideas behind its proof in Section \ref{s:fractionalintro} further below.

Finally, we remark that a variant of the Fourier-Jacobi expansion trick sketched at the beginning of this subsection has been recently developed by B\"ocherer and Das to prove non-vanishing of fundamental Fourier coefficients  of Siegel modular forms of degree $n$   \cite{bocherer-das-2020}. By using their variant, it seems plausible that the methods of this paper may allow one to extend Theorems A and B above to Siegel cusp forms of higher degree. We do not pursue this extension here.
\subsection{Central $L$-values for dihedral twists of spin $L$-functions}\label{s:lvaluesintro}

For two matrices $S_1$, $S_2$ in $\Lambda_2$,  write $S_1 \sim S_2$ if there exists $A \in \SL_2(\Z)$ such that $S_1 = {}^t\!A S_2A$. Let $F \in S_k(\Gamma_0^{(2)}(N))$ with $k>2$ even and $N$ odd and squarefree.  Using the defining relation for Siegel cusp forms, we see that \begin{equation}\label{siegelinv}
a(F, S_1) = a(F,S_2)\qquad\text{if }S_1\sim S_2,
\end{equation}
thus showing that $a(F, S)$ only depends on the $\SL_2(\Z)$-equivalence class of the matrix $S$. Let $d < 0$ be a fundamental discriminant,    let $\Cl_K$ denote the ideal class group of $K= \Q(\sqrt{d})$, and let $w(K) \in \{2,4,6\}$ be the number of roots of unity in $K$.  It is well-known that the $\SL_2(\Z)$-equivalence classes of matrices in $\Lambda_2$ of discriminant $d$ are in natural bijective correspondence with the elements of $\Cl_K$. So, for any character $\Lambda$ of the finite group $\Cl_K$, we can define
\begin{equation}\label{rfdeflam}
B(F, \Lambda) = \sum_{S \in \Cl_K}a(F, S) \Lambda(S),
\end{equation}
which may be viewed as a Bessel period \cite[Prop. 3.5]{DPSS15}.

The space $S_k(\Gamma_0^{(2)}(N))$ has a natural subspace  $S_k(\Gamma^{(2)}_0(N))^{\rm CAP}$ spanned by the \emph{Saito-Kurokawa lifts}. If $F$ is a Saito-Kurokawa lift then $a(F,S)$ (for fundamental $S$) depends \emph{only} on $d=\disc(S)$ and is fairly well-understood. In particular, for $F \in S_k(\Gamma^{(2)}_0(N))^{\rm CAP}$, the Bessel period $B(F, \Lambda)$ vanishes whenever $\Lambda \neq 1_K$, where
$1_K$ denotes the trivial character of $\Cl_K$. Now suppose that $F$ is not a Saito-Kurokawa lift. Let $\phi$ be the adelization of $F$, and suppose that $\phi$ generates an irreducible automorphic representation $\pi$ of $\GSp_4(\A)$. B\"ocherer \cite{boch-conj} made the remarkable conjecture that $|B(F, 1_K)|^2 = A_F \cdot w(K)^{2} \cdot |d|^{k-1} \cdot L(\tfrac12, \pi \otimes \chi_{d})$ where $\chi_d$ is the quadratic character associated to $K/\Q$ and $A_F$ is a constant depending only on $F$.

More generally, let $\AI(\Lambda)$  be the automorphic representation of $\GL(2,\A)$ given by the automorphic induction of $\Lambda$ from $K$; it is generated by (the adelization of)
the dihedral modular form  $\theta_{\Lambda}(z) = \sum_{0 \ne \mathfrak{a} \subset \mathcal O_K}\Lambda(\mathfrak{a}) e^{2 \pi i N(\mathfrak{a})z}$ of weight 1. It is easy to check that $L(s, \pi \otimes \AI(\Lambda)) = L(s, \pi)  L(s, \pi \otimes \chi_{d})$. \emph{Now, assume that the refined Gan-Gross-Prasad conjecture \up{see \cite[Conjecture 1.12]{DPSS15} and \cite[(1.1)]{yifengliu}} for the pair $(\phi, \Lambda)$ holds true.} In fact, this conjecture for $\Lambda=1_K$ is now known thanks to work of Furusawa and Morimoto  \cite{furmori} (which combined with \cite{DPSS15} completes the proof of B\"ocherer's conjecture)  and the proof for general $\Lambda$ has been recently announced by the same authors. Then Theorem 1.13 of \cite{DPSS15} implies that under some mild assumptions,  \begin{equation}\label{bochrefinedeq}|B(F, \Lambda)|^2 =  c_F w(K)^2 \ |d|^{k-1}L(\tfrac12, \pi \times  \AI(\Lambda)),
 \end{equation} where $c_F$ is an explicit non-zero constant depending only on $F$ and $L(s, \pi \times  \AI(\Lambda))$ is the tensor product $L$-function of the spin (degree 4) $L$-function of $\pi$ and the standard (degree 2) $L$-function of $\AI(\Lambda)$. We show in Proposition \ref{p:ggpimpliesG} that a variant of \eqref{bochrefinedeq}, where the equality is replaced by an inequality, holds in a more general setup (assuming the refined Gan-Gross-Prasad conjecture).

 The identities \eqref{rfdeflam} and \eqref{bochrefinedeq}  demonstrate that the fundamental Fourier coefficients of Hecke eigenforms in $S_k(\Gamma_0^{(2)}(N))$ are intimately connected with central $L$-values of the degree 8 $L$-function $L(s, \pi \times  \AI(\Lambda))$ as $\Lambda$ varies over the ideal class characters of $K$. By inverting \eqref{rfdeflam}, we can write
 \begin{equation}\label{e:invertedfund}  a(F,S) = \frac{1}{|\Cl_K|}\sum_{\Lambda \in \widehat{\Cl_K}} B(F, \Lambda) \Lambda^{-1}(S), \end{equation} which expresses each fundamental $a(F,S)$ as a weighted average of the Bessel periods  $B(F, \Lambda)$.

  Now, combining \eqref{bochrefinedeq} and \eqref{e:invertedfund} with Theorem B, we obtain the following corollaries.

  \begin{corollary}\label{c:lvalueslarge} Let $\pi$ be a cuspidal automorphic representation of $\GSp_4(\A)$ that is not of Saito-Kurokawa type, such that $\pi$ arises from a form in $S_k(\Gamma_0^{(2)}(N))$ with $k>2$ even and $N$ odd and squarefree. Fix $\eps>0$. Assume the refined Gan--Gross-Prasad conjecture \cite[Conjecture 1.12]{DPSS15}. Then for all sufficiently large $X$, there are $\ge X^{1-\eps}$ negative fundamental discriminants $d$ with $|d|\asymp X$ such that for $K=\Q(\sqrt{d})$,
  \[ \tst \frac{1}{|\Cl_K|}\sum_{\Lambda \in \widehat{\Cl_K}} L(\tfrac12,  \pi \times  \AI(\Lambda)) \gg_\pi |d|^{-1/2} \exp\left(\frac{1}{82} \sqrt{\frac{\log |d|}{\log \log |d|}} \right).
  \]
  \end{corollary}

By specializing further to the case of Yoshida lifts, we obtain the following application which is purely about central $L$-values of dihedral twists of \emph{classical} newforms.
\begin{corollary}\label{c:yoshidavalues}Let $k>2$ be an even integer. Let $N_1$, $N_2$ be two positive, squarefree integers such that $M = \gcd(N_1, N_2)>1$. Let $f$ be a holomorphic newform of weight $2k-2$ on $\Gamma_0(N_1)$ and $g$ be a holomorphic newform of weight $2$ on $\Gamma_0(N_2)$. Assume that for all primes $p$ dividing $M$ the Atkin-Lehner eigenvalues of $f$ and $g$ coincide. Fix $\eps>0$. Then for all sufficiently large $X$, there are $\ge X^{1-\eps}$ negative fundamental discriminants $d$ with $|d|\asymp X$ with the property that there exists  an ideal class group character $\Lambda$ of $K=\Q(\sqrt{d})$  such that
\[ \tst L(\tfrac12, f \times \AI(\Lambda))  L(\tfrac12, g \times \AI(\Lambda)) \gg_{f,g} |d|^{-1/2} \exp\left(\frac{1}{82} \sqrt{\frac{\log |d|}{\log \log |d|}} \right).\]
   \end{corollary}
   Corollary \ref{c:yoshidavalues} strengthens the main theorem of \cite{sahaschmidt} which showed the existence of $\Lambda$ with (simultaneous) non-vanishing for  $L(\frac12, f \times \AI(\Lambda))$ and $L(\frac12, g \times \AI(\Lambda))$ and remarked: ``while our method gives a lower bound
on the number of non-vanishing twists, it does not give a lower bound on the
size of the non-vanishing $L$-value itself." Corollary \ref{c:yoshidavalues} successfully achieves this.

  \subsection{Fractional moments of $L$-values}\label{s:fractionalintro}
  Combining \eqref{bochrefinedeq} and  \eqref{e:invertedfund}, we can write
  \begin{equation}\label{e:invertedggp}|a(F,S)| \ll_F |d|^{\frac{k}{2}-\frac12}\frac{1}{|\Cl_K|} \sum_{\Lambda \in \widehat{\Cl_K}} \sqrt{L(\tfrac12, \pi \times  \AI(\Lambda))}.\end{equation}

 From the above, we see that the Generalized Lindel\"of hypothesis for $L(\tfrac12, \pi \times  \AI(\Lambda))$ implies that $|a(F,S)| \ll_{F, \eps} |\disc(S)|^{\frac{k}2 - \frac{1}{2} + \eps},$ which is still quite far from \eqref{e:ressald}.

Therefore, in order to prove Theorem C, we need to go beyond the bound obtained by a naive application of the Generalized Lindel\"of hypothesis. We do this by using Soundararajan's method \cite{Soundararajan} for bounding moments of $L$-functions. Assuming GRH, we prove the following bound (Theorem \ref{thm:momentbd}) which, thanks to \eqref{e:invertedggp}, implies Theorem C:
\[
\frac{1}{|\tmop{Cl}_K|}  \sum_{\Lambda \in \widehat{ \tmop{Cl}_K}} \sqrt{L(\tfrac12, \pi \times \AI(\Lambda))} \ll_\varepsilon \frac{1}{(\log |d|)^{\frac18-\varepsilon}}.
\]
The main contribution to the moments of $L(\tfrac12, \pi \times \AI(\Lambda))$ will come from its large values and we expect that these should be approximated by the large values of
$
\exp( \sum_{p^n < |d|} \frac{b_{\pi  \times \AI(\Lambda)}(p^n)}{ p^{n/2}}),
$
where $b_{\pi  \times \AI(\Lambda)}(n)$ is the $n$'th coefficient of the Dirichlet series of $\log L(s,\pi  \times \AI(\Lambda))$. For ease of discussion let us assume here that $d$ is prime, $N=1$, and $\pi$ transfers to a \emph{cuspidal} representation of $\GL_4$\footnote{If this is not the case (e.g., if $\pi$ corresponds to a Yoshida lift) the estimates below will be slightly different and the resulting bound for the moment predicted by the heuristic will also differ.}. Separately analyzing the primes, squares of primes and higher prime powers we show under GRH
\[
\sum_{p^n < |d|} \frac{b_{\pi  \times \AI(\Lambda)}(p^n)}{ p^{n/2}}= \sum_{p < |d|} \frac{b_{\pi}(p)b_{\AI(\Lambda)}(p)}{\sqrt{p}}-\frac{1}{2} \log \log |d|(1+o(1))
\]
where $b_{\pi}(p),b_{\AI(\Lambda)}(p)$ respectively denote the $p$'th coefficient of the Dirichlet series of $\log L(s,\pi)$ and $ \log L(s, \AI(\Lambda))$. For primes with $(\frac{d}{p})=1$ so that $p\mathcal O_K=\mathfrak p \overline {\mathfrak p}$, as $\Lambda$ varies over $\widehat{\tmop{Cl}_K}$, we expect that $b_{\AI(\Lambda)}(p)=\Lambda(\mathfrak p)+\Lambda(\mathfrak p)^{-1}$ behaves like the random variable $X_p+X_p^{-1}$ where $\{X_p\}_p$ are iid random variables uniformly distributed on the unit circle (if $(\frac{d}{p})=-1$, $b_{\AI(\Lambda)}(p)=0$). Consequently the sum on the r.h.s. above is modelled by the random variable  $\sum_{p < |d| \\  } \frac{b_{\pi}(p) (X_p+X_p^{-1})}{\sqrt{p}} 1_{(\frac{d}{p})=1}$, which can be shown to have a normal limiting distribution as $d \rightarrow \infty$ with mean $0$ and variance $2\sum_{\substack{p < |d| \\ }} \frac{b_{\pi}(p)^2 }{p} 1_{(\frac{d}{p})=1} \sim  \log \log |d|$, which we prove under GRH. The preceding discussion suggests
\[
\begin{split}\tst
\frac{1}{|\tmop{Cl}_K|}  \sum_{\Lambda \in \widehat{ \tmop{Cl}_K}} \sqrt{L(\tfrac12, \pi \times \AI(\Lambda))}
& \tst \asymp (\log |d|)^{-1/4} \mathbb E \left( \exp\left( \frac{1}{2} \sum_{p < |d|} \frac{b_{\pi}(p)(X_p+X_p^{-1}) 1_{(\frac{d}{p})=1}}{\sqrt{p}}   \right)\right) \\
& \asymp \tst  (\log |d|)^{-1/8}
\end{split}
\]
where in the last step we have used that the moment generating function of a normal random variable $X$ with mean $0$ and variance $\sigma^2$ is given by $\mathbb E (e^{zX})=e^{\frac12 z^2 \sigma^2}$. Remarkably, Soundararajan's method allows us to make this heuristic argument rigorous for the upper bound, up to the loss of a factor $(\log |d|)^{\varepsilon}$, which occurs due to a sub-optimal treatment of the large primes.


\subsection{Notations}\label{sec:notations}We use the notation
$A \ll_{x,y,z} B$
to signify that there exists
a positive constant $C$, depending at most upon $x,y,z$,
so that
$|A| \leq C |B|$.
 The symbol $\varepsilon$ will denote a small positive quantity. We write $A(x) = O_y(B(x))$ if there exists a positive real number $M$ (depending on $y$) and a real number $x_0$ such that $|A(x)| \le M |B(x)|$ for all $x \ge x_0$.

For a positive integer $n$ with prime factorization $n = \prod_{i=1}^k p_i^{\alpha_i}$, we define $\omega(n)=k$, $\Omega(n)=\sum_{i=1}^k \alpha_i$. We let $\mu(n)$ denote the M\"obius function, i.e., $\mu(n)=(-1)^{\omega(n)}$ if $\omega(n)=\Omega(n)$, and $\mu(n) =0$ otherwise. We say that $n$ is squarefree if $\mu(n) \neq 0$. We let $(a,b)$ or $\gcd(a,b)$ denote the greatest common divisor of $a$ and $b$.

We say that $d$ is a fundamental discriminant if $d$ is the discriminant of the field $\Q(\sqrt{d})$. For a fundamental discriminant $d$, we let $\chi_d$ be the associated quadratic Dirichlet character. Given any representation $\pi$ of a group, we let $\hat{\pi}$ denote the contragredient, and $V_\pi$ denote the representation space.
We use $\A$ to denote the ring of adeles over $\Q$ and we use $\A_F$ to denote the ring of adeles over $F$ for a general number field $F$. If $G$ is a reductive group such that the local Langlands correspondence is known for each $G(F_v)$ and $\pi$ is an automorphic representation of $G(\A_F)$, then we formally (as an Euler product over finite places) define the $L$-function $L(s, \rho(\pi)):= L(s, \pi, \rho)$ for each finite dimensional representation $\rho$ of the dual group. All $L$-functions in this paper will denote the finite part of the $L$-function (i.e., without the archimedean factors), so that for a number field $F$ and an automorphic representation $\pi$ of $\GL_n(F)$, we have $L(s, \pi) = \prod_{v<\infty} L(s, \pi_v)$. All $L$-functions will be normalized to take $s \mapsto 1-s$. For an integer $N$ we denote $L^N(s, \pi)=\prod_{v \nmid N}L(s, \pi_v)$.
Given a reductive group $G$ and two irreducible automorphic representations $\pi = \otimes_v \pi_v$ and $\sigma = \otimes_v \sigma_v$ of $G(\A_F)$, we say that $\pi$ and $\sigma$ are nearly equivalent if $\pi_v \simeq \sigma_v$ for all but finite many places $v$ of $F$.
\subsection{Acknowledgements}We thank Ralf Schmidt for helpful discussions concerning the material in Section \ref{s:lfunctionssiegel}. We thank Valentin Blomer and Farrell Brumley for forwarding us their preprint \cite{Blomer-Brumley}. We thank the anonymous referee for useful comments and corrections which improved this paper. This work was supported by the Engineering and Physical Sciences Research Council [grant number EP/T028343/1].
\section{Preliminaries on half integral weight forms}
\label{halfintmainsec}
The goal of this section is to set up some notation and lay out some key properties concerning cusp forms of half-integral weight on the complex upper half-plane.
\subsection{Notations}\label{halfintsec-notations}
The group $\SL_2(\R)$ acts on the upper half-plane $\H$ by $\gamma z = \frac{az+b}{cz+d},$ where $\gamma = \mat{a}{b}{c}{d} $ and $z=x+iy $. For a positive integer $N$, let $\Gamma_0(N)$ denote the congruence subgroup consisting of matrices $\mat{a}{b}{c}{d}$ in $\SL_2(\Z)$ such that $N$ divides $c$. For a complex number $z$, let $e(z)$ denote $e^{2\pi i z}$.

Let $\theta(z) = \sum_{n = -\infty}^\infty e(n^2 z)$ be the standard theta function on $\H$. If $A = \mat{a}{b}{c}{d} \in \Gamma_0(4)$, we have $\theta(Az) = j(A, z)\theta(z),$ where $j(A, z)$ is the so-called $\theta$-multiplier. For an explicit formula for $j(A, z)$, see~\cite[(1.10)]{Shimura1}. Let $S_{k+\frac{1}{2}}(4N)$ denote the space of holomorphic cusp forms of weight $k+\frac{1}{2}$  for the group $\Gamma_0(4N)$. In other words, a function $f : \H \rightarrow \C$ belongs to $S_{k+\frac{1}{2}}(4N)$ if
\begin{enumerate}

\item $f(Az) =  j(A, z)^{2k +1} f(z)$ for every $A = \mat{a}{b}{c}{d} \in \Gamma_0(4N)$,
\item $f$ is holomorphic,
\item $f$ vanishes at the cusps.

\end{enumerate}

Any $f \in  S_{k+\frac{1}{2}}(4N)$ has the Fourier expansion
$$\tst f(z) = \sum_{n > 0} a(f, n)e(nz).$$ We let $c(f,n)$ denote the ``normalized" Fourier coefficients, defined by
 $$c(f,n) = a(f,n)n^{ \frac14-\frac k2}.$$

For $f, g \in S_{k+\frac12}(4N)$, we define the Petersson inner product $\langle f, g\rangle$ by $$\tst \langle f, g\rangle = [\SL_2(\Z): \Gamma_0(4N)]^{-1} \int_{\Gamma_0(N) \bs \H}f(z) \overline{g(z)} y^{k + \frac12} \frac{dx dy}{y^2}.$$

\subsection{The Kohnen plus space and decomposition into old and newspaces}
Fix positive integers $k$, $N$ such that  $N$ is odd and squarefree. We recall the definition of the Kohnen plus space  $S^+_{k+\frac12}(4N) \subseteq S_{k+\frac12}(4N).$  The space $S^+_{k+\frac12}(4N)$ consists of all forms $f$ in  $S_{k+\frac12}(4N)$ for which $a(f,n) = 0$ whenever $n \equiv (-1)^{k+1}$ or $2 \mod{4}$.  According to the results of \cite{KohnenNewform}, there exists a canonically defined subspace $S^{+, \new}_{k+\frac12}(4N) \subset S^+_{k+\frac12}(4N)$ and a decomposition
\begin{equation}\label{e:oldnew}\tst S^+_{k+\frac12}(4N) = \bigoplus_{\substack{r, \ell \ge 1 \\ r \ell|N}} S^{+, \new}_{k+\frac12}(4\ell)|U(r^2),\end{equation} where we define \begin{equation}\label{defur2}\tst f|U(r^2) = r^{\frac12 - k}\sum_{n >0 }a(f,r^2n)e(nz).\end{equation} It is known~\cite[Prop. 1.5]{Shimura1} that if $(r, \ell)=1$, then $U(r^2)$ takes $S_{k + \frac12}(4\ell)$ to $S_{k + \frac12}(4r\ell)$. It is also useful to note that \begin{equation}\label{defur3}c(f|U(r^2), n) = c(f, r^2 n).\end{equation}

\subsection{Hecke operators and the Shimura correspondence}
For all primes $p$ coprime to $N$ there exist Hecke operators $T(p^2)$ acting on the space $S_{k+\frac12}(4N)$; see~\cite[Thm. 1.7]{Shimura1}. A \emph{newform} in $S^{+, \new}_{k+\frac12}(4N)$ is defined to be an element of $S^{+, \new}_{k+\frac12}(4N)$ that is an eigenfunction of the Hecke operators $T(p^2)$ for $p\nmid N$. The newforms are uniquely determined up to multiplication by non-zero complex numbers and are in fact also eigenforms for the operators $U(p^2)$ for all $p|N$  \cite[Theorem 2]{KohnenNewform}. The space $S^{+, \new}_{k+\frac12}(4N)$ has an orthogonal basis consisting of newforms.

According to the Shimura lifting \cite{Shimura1}, as refined by Kohnen in \cite{KohnenNewform}, there is an isomorphism \begin{equation}\label{e:heckeiso}S^{+, \new}_{k+\frac12}(4N)  \overset{\simeq}{\rightarrow} S_{2k}^{\new}(N)\end{equation} as \emph{Hecke modules}, where $S_{2k}^{\new}(N)$ is the orthogonal complement of the space of cuspidal oldforms of weight $2k$ for $\Gamma_0(N)$ as defined by Atkin--Lehner \cite{AtkinLehner}. The Shimura lifting \eqref{e:heckeiso} takes each newform in $S^{+, \new}_{k+\frac12}(4N)$ (as defined above) to a newform (in the sense of Atkin--Lehner \cite[Lemma 18]{AtkinLehner}) in $S_{2k}^{\new}(N)$ with the same Hecke eigenvalues. More precisely, if $f \in S^{+, \new}_{k+\frac12}(4N)$ is a newform, and $g \in  S_{2k}^{\new}(N)$ is the Shimura lift of $f$ according to \eqref{e:heckeiso}, then for each prime $p \nmid N$ there exists a real number $\lambda_f(p) \in [-2, 2]$ (by Deligne's bound for the normalized Hecke eigenvalue) such that \[f|T(p^2) = \lambda_f(p) p^{k - \frac12} f, \quad g|T(p) = \lambda_f(p) p^{k - \frac12} g.\]

In view of \eqref{e:oldnew} and the fact that the operators $U(p)$ with $p|N$ commute with $T(p^2)$, $p \nmid N$, a basis of $S^{+}_{k+\frac12}(4N)$ consisting of eigenforms for $T(p^2)$, $p\nmid N$, is given by \begin{equation}\label{e:heckebasis}\tst \B_{k+\frac12,4N} = \bigcup_{\substack{r, \ell \ge 1 \\ r \ell|N}} \{f|U(r^2): f\in \B^\new_{k+\frac12,4\ell}\},
\end{equation} where $\B^\new_{k+\frac12,4\ell}$ is an (orthogonal) basis of  $S^{+, \new}_{k+\frac12}(4\ell)$ consisting of newforms. Note however that all members of $\B_{k+\frac12,4N}$ are not necessarily orthogonal to each other.
The following result will be useful for us; recall the definitions of $\Omega(n)$ and $\omega(n)$ from Section \ref{sec:notations}.

\begin{lemma}\label{l:fourierUr2} Let $r$, $\ell$ be positive, odd, squarefree integers with $(r, \ell)=1$ and let $f \in S^{+, \new}_{k+\frac12}(4\ell)$ be a newform. Then for any odd squarefree integer $n$, putting $d=(-1)^kn$, we have \[\tst c(f|U(r^2), n) = c(f, r^2n) = c(f,n) \prod_{p|r}\left(\lambda_f(p) - \frac{1}{\sqrt{p}}\left(\frac{d}{p} \right) \right).\]
Additionally, for any odd integer $r \ge 1$ with $(r,\ell)=1$ we have
\[
|c(f,r^2n)| \le 3^{\Omega(r)} |c(f,n)|.
\]
\end{lemma}

\begin{proof}The first statement follows from \cite[Corollary 1.8 (i)]{Shimura1}. Using that $|\lambda_f(p)|\le 2$ and applying \cite[Corollary 1.8 (ii)]{Shimura1} we will establish the second claim  by the following simple induction argument.
It suffices to show for each $p\nmid 2\ell$ that
\[
|c(f,p^{2m} n)| \le 3^m |c(f,n)|.
\]
The case $m=0$ is trivial and $m=1$ follows from the first claim of the lemma.
By \cite[Corollary 1.8 (ii)]{Shimura1} we have for any $m \ge 1$
\[
c(f,p^{2m+2}n) =\lambda_f(p) c(f,p^{2m}n)-(-1)^{k(p^2-1)/2} c(f,p^{2m-2} n).
\]
Hence, for $m \ge 1$ we get that
\[
\begin{split}
|c(f,p^{2m+2}n)| &\le 2 |c(f,p^{2m}n)|+|c(f,p^{2m-2 }n)|\\
& \le  (2 \cdot 3^{m}+3^{m-1}) |c(f,n)|\le 3^{m+1} |c(f,n)|.
\end{split}\]\end{proof}

\subsection{An explicit version of Waldspurger's formula}A well-known formula of Waldspurger \cite{Waldscentral} that was refined and made explicit in special cases by Kohnen \cite{Kohnen85}, expresses the squares of Fourier coefficients of half-integral weight eigenforms in terms of central $L$-values. We state a version of it below for elements of the basis \eqref{e:heckebasis}.
\begin{proposition}\label{p:walds}Let $r$, $\ell$ be positive, odd, squarefree integers with $(r, \ell)=1$. Let $f$ be a newform in  $S^{+, \new}_{k+\frac12}(4\ell)$ and let $g \in S_{2k}^{\new}(\ell)$ be the Shimura lift of $f$. Then, for any squarefree positive integer $n$ with $(n, 4\ell)=1$, and $d=(-1)^kn$, we have \[\tst
\frac{|c(f|U(r^2), n)|^2}{\langle f, f\rangle} = 2^{\omega(\ell)} \frac{L(\frac12, g \otimes \chi_d )}{\langle g, g \rangle } \frac{(k-1)!}{\pi^k} \left(\prod_{p|r}\left(\lambda_f(p) - \frac{1}{\sqrt{p}}\left(\frac{d}{p} \right) \right)\right)^2  \]
provided that \begin{enumerate}
\item $d \equiv 1 \pmod{4}$
\item for each prime $p|\ell$, we have $\left(\frac{d}{p}\right)=w_p$, where $w_p$
is the eigenvalue for the  Atkin-Lehner operator at $p$ acting on $g$.
\end{enumerate}
If either of the two conditions above are not met, then $c(f|U(r^2), n) = 0$.
\end{proposition}
\begin{proof}This follows from Corollary 1 of \cite{Kohnen85}, Proposition 4 of \cite{KohnenNewform},  and Lemma \ref{l:fourierUr2}.
\end{proof}
\subsection{Estimates on moments of Fourier coefficients}
\begin{proposition} \label{prop:moments}
Let $f \in S_{k+\frac12}^+(4N)$ where $N$ is odd and squarefree. Then there exists $M\ge 2$ such that for all sufficiently large $X$,
  \begin{equation} \label{eq:RS}
  \sum_{\substack{X \leq n \leq MX \\ (n,2N)=1 }} |c(f,n)|^2 \mu^2(n) \asymp_{f, M} X
   \end{equation}
   and for any $\varepsilon>0$
   \begin{equation} \label{eq:4th}
         \sum_{\substack{X \leq n \leq 2X  \\ (n,2N)=1}} |c(f,n)|^4 \mu^2(n) \ll_{f, \eps} X^{1+\varepsilon}.
   \end{equation}
 \end{proposition}
\begin{proof}We first prove \eqref{eq:RS}. For the upper bound we use that $y^{k+1/2}|f(z)|$ is bounded on $\mathbb H$ and hence we have
 \begin{equation} \label{eq:parseval}\tst
 \begin{split}
 \sum_{\substack{X \le n \le 2X \\ (n,2N)=1}} |c(f,n)|^2 \mu^2(n) \le \sum_{X \le n \le 2X} |c(f,n)|^2 \ll & \tst \frac{1}{X^{k-1/2}} \sum_{n} |c(f, n)|^2 n^{k - \frac12} e^{-4\pi n /X}\\
 =& \tst \frac{1}{X^{k-1/2}} \int_0^1 |f(x+i/X)|^2 dx \ll_f X.
 \end{split}
 \end{equation} For the lower bound, we use a result obtained in the proof of \cite[Proposition 3.7]{saha-2013}, which gives for any $M \ge 1$ that
 \[
 \sum_{\substack{ (n,2N)=1}} |c(f, n)|^2 \mu^2(n)  e^{-n/(X\sqrt{M})} \gg_f X\sqrt{M}.
 \]
 Using \eqref{eq:parseval} along with partial summation we can bound the tail end of the sum as follows
 \[
  \sum_{\substack{n \ge M X\\ (n,2N)=1}} |c(f, n)|^2 \mu^2(n) e^{-n/(X\sqrt{M})} \ll_f M e^{-\sqrt{M}} X.
 \]
Combining the two bounds above we have for $M$ sufficiently large that
 \[\tst
 \begin{split}
   \sum_{\substack{n \le M X\\ (n,2N)=1}} |c(f,n)|^2 \mu^2(n)   & \ge    \sum_{\substack{n \le MX\\ (n,2N)=1}} |c(f, n)|^2 \mu^2(n)  e^{-n/(X\sqrt{M})} \\
   &= \sum_{\substack{n \ge 1\\ (n,2N)=1}} |c(f, n)|^2 \mu^2(n) e^{-n/(X\sqrt{M})} +O_f(M e^{-\sqrt{M}} X)
   \gg_f X \sqrt{M}.
   \end{split}
 \]
 Finally, we note that by \eqref{eq:parseval} the contribution from terms to the l.h.s. above with $n \le X$ is $O_f(X)$, which completes the proof of the lower bound in \eqref{eq:RS}.

For the proof of \eqref{eq:4th}, we use \eqref{e:heckebasis} to reduce to the case $f = f_1|U(r^2)$ where $f_1 \in S^{+, \new}_{k+\frac12}(4\ell)$ is a newform with $r\ell|N$. Using Proposition \ref{p:walds}, it now suffices to prove that \[
 \sumfund_{|d|\le X} L(\tfrac12, g \otimes \chi_d)^2 \ll_{g, \eps} X^{1+\varepsilon}
 \] where $g$ is the Shimura lift of $f_1$ and the sum is over fundamental discriminants $d$. This follows from the approximate functional equation and Heath-Brown's quadratic large sieve \cite{HB}, using a straightforward modification of the proof of \cite[Theorem 2]{HB} (see also \cite[Corollary 2.5]{SoundYoung}).
\end{proof}

\section{Sign changes for coefficients of half-integral weight forms}\label{s:signs}

\subsection{Statement of main result}
Throughout this section let $k \ge 2$ be an integer and $N \ge 1$ be odd and squarefree.
The main theorem to be proved in this section is
\begin{theorem} \label{thm:signs}
  Let $f \in S_{k+\frac12}^+(4N)$ be a fixed cusp form whose Fourier coefficients $c(f,n)$ are all real. Then there exists $M \ge 2$ such that given any $\varepsilon>0$,
  the sequence $\{ c(f,n) \}_{\substack{X \leq n \leq MX \\ 2n \text{ is  squarefree}\\ (n,N)=1}}$ has at least $ \gg_{f,M,\varepsilon} X^{1-\varepsilon}$ sign changes.
\end{theorem}
The main novelty here is that this result holds for all cusp forms $f \in S_{k+1/2}^+(4N)$, not just Hecke eigenforms, and this is crucial for our later application. Previously it was not apparently even known that there are infinitely many sign changes of $c(f,n)$ as $n$ ranges over squarefree integers for $f \in S_{k+\frac12}^+(4N)$.

Our proof builds upon the methods developed in \cite{MatomakiRadziwill, LesterRadziwill2019signs} and relies upon the following two propositions. The first of which shows that the size of $|c(f,n)|$ is relatively well-behaved for most short intervals $[x,x+y]$.

\begin{proposition} \label{prop:lowerbd}
 Let $f \in S_{k+\frac{1}{2}}^+(4N)$.
There exists $M \ge 2$ such that given any $\varepsilon >0$ and $2 \le y \le X/2$ there are $\gg_{f,M,\varepsilon} X^{1-\frac32 \varepsilon}$ integers $X \le x \le MX$
such that
  $$
  \sum_{\substack{x \leq n \leq x + y  \\ (n,2N)=1}} |c(f, n)| \mu^2(n)  >  \frac{y}{X^{\varepsilon}}.
  $$

 \end{proposition}

 Our other main proposition shows that we can obtain square-root cancellation in sums of $c(f,n)$ over almost all short intervals $[x,x+y]$.

\begin{proposition} \label{prop:cancellation}
 Let $f \in S_{k+\frac{1}{2}}^+(4N)$.
Then for $1 \leq y \leq X^{\frac{1}{205}}$ we have that
 \[
  \sum_{X \le x \le 2X} \bigg|\sum_{\substack{x \le n \le x+y  \\ (n,2N)=1}} c(f,n) \mu^2(n) \bigg| \ll_f X \sqrt{y}.
 \]
\end{proposition}

We will now prove Theorem \ref{thm:signs} using Propositions \ref{prop:lowerbd} and \ref{prop:cancellation}. The proof of Propositions \ref{prop:lowerbd} and \ref{prop:cancellation} will be given in Sections \ref{sec:lowerbd} and \ref{sec:cancellation}, respectively.

\begin{proof}[Proof of Theorem \ref{thm:signs}]
Observe that if the Fourier coefficients $c(f,n)$ are real and
\[
  \bigg | \sum_{\substack{x \leq n \leq x + y \\ (n,2N)=1}} c(f,n) \mu^2(n) \bigg |  <  \sum_{\substack{x \leq n \leq x + y \\ (n,2N)=1}} |c(f,n)| \mu^2(n)
\]
then the interval $[x,x+y]$ must contain a sign change of $c(f,n)$ where $n\in[x,x+y]$ ranges over odd squarefree integers that are coprime to $N$.
We will show that for most integers $X \le x \le MX$ that the above inequality holds for intervals of length $y = X^{6 \varepsilon}$.

By Chebyshev's inequality, the number of integers $X \le x \le MX$ for which
  \begin{equation}\label{eq:cancellation}\tst
  \bigg | \sum_{\substack{x \leq n \leq x + y \\ (n,2N)=1}} c(f,n) \mu^2(n) \bigg | \leq \frac{y}{X^{\varepsilon}}
 \end{equation}
does not hold is
  \[ \tst
  \le \frac{X^{\varepsilon}}{y}\sum_{X \le x \le MX} \bigg| \sum_{\substack{x \leq n \leq x + y \\ (n,2N)=1}} c(f,n) \mu^2(n) \bigg| \ll_{f,M} \frac{X^{1+\varepsilon}}{\sqrt{y}}=X^{1-2\varepsilon},
  \]
  where we have used Proposition \ref{prop:cancellation} in the last inequality.
By Proposition \ref{prop:lowerbd} we have that
  \begin{equation}\label{eq:lowerbd} \tst
  \sum_{\substack{x \leq n \leq x + y \\ (n,2N)=1}} |c(f,n)| \mu^2(n) > \frac{y}{X^{\varepsilon}}
\end{equation}
  for all integers $X \le x \le MX$ outside an exceptional set of size  $\ll_{f,M, \eps} X^{1 - 3\varepsilon/2}$.
  Hence, there exist at least $\gg_{f, M, \eps} X^{1 - 3 \varepsilon/2}$ integers $X \le x \le MX$ such that \eqref{eq:cancellation} and \eqref{eq:lowerbd} hold. Therefore, we obtain at least $ \gg_{f, M, \eps} \frac{X^{1-3\varepsilon/2}}{y}=X^{1 - 15 \varepsilon/2}$ sign changes of $c(f,n)$ along integers $X \le n \le MX$, that are odd, squarefree, and coprime to $N$.
  \end{proof}




\subsection{Proof of Proposition \ref{prop:lowerbd}}

 We first prove the following result, which is an easy consequence of Proposition \ref{prop:moments}.
 \begin{lemma} \label{lem:keybd2}
 Let $f \in S_{k+\frac12}^{+}(4N)$. Then there exists $M \ge 2$ such that given any $\varepsilon>0$
\begin{equation}
\sum_{\substack{X \le n \le MX \\  (n,2N)=1 \\  |c(f,n)| \le X^{\varepsilon}}} |c(f,n)| \mu^2(n) \gg_{f,M,\varepsilon} X^{1-\varepsilon/2}.
\end{equation}
 \end{lemma}
 \begin{proof}
Applying H\"older's inequality gives
\[ \tst
\sum_{\substack{X \leq n \leq MX \\ (n,2N)=1}} |c(f, n)|^2 \mu^2(n) \le \Bigg( \sum_{\substack{X \leq n \leq MX \\ (n,2N)=1}} |c(f,n)| \mu^2(n) \Bigg)^{2/3} \Bigg(  \sum_{\substack{X \leq n \leq MX \\ (n,2N)=1}} |c(f,n)|^4 \mu^2(n) \Bigg)^{1/3}.
\]
Hence, using Proposition \ref{prop:moments} we conclude that
\begin{equation} \label{eq:holder2}
\sum_{\substack{X \le n \le MX \\  (n,2N)=1}} |c(f,n)| \mu^2(n) \gg_{f, M, \eps} X^{1-\varepsilon/2}.
\end{equation}
Also, by \eqref{eq:RS} we have that
\[
\sum_{\substack{X \le n \le MX \\  (n,2N)=1 \\  |c(f,n)|  > X^{\varepsilon}}} |c(f,n)| \mu^2(n) < X^{-\varepsilon} \sum_{\substack{X \le n \le MX \\  (n,2N)=1}} |c(f,n)|^2  \mu^2(n)\ll_{f,M, \eps} X^{1-\varepsilon}.
\]
Combining this with \eqref{eq:holder2} completes the proof.
 \end{proof}

\begin{proof}[Proof of Proposition \ref{prop:lowerbd}] \label{sec:lowerbd}
 Let $C(f,n)=|c(f,n)| \mu^2(n) 1_{(n,2N)=1} 1_{|c(f,n)| \le X^{\varepsilon}}$. Applying Lemma \ref{lem:keybd2}, we see that
\[
X^{1-\varepsilon/2} \ll_{f,M, \eps} \sum_{\substack{ X+y \le n \le 2MX }} C(f,n)  \le \frac{1}{y} \sum_{X \le x \le 2MX} \sum_{\substack{ x \le n \le x+y}} C(f,n) ,
\]
where the second inequality above follows since every term in the sum on the l.h.s. is counted $\lfloor y\rfloor +1$
times on the r.h.s.. Let $S=\{X \le x \le 2MX : \sum_{x \le n \le x+y} C(f,n) \le y/X^{\varepsilon}\}$. The contribution to the r.h.s. from the integers $x \in S$ is $\ll_{f,M,\varepsilon}\frac{1}{y} \cdot X \frac{y}{X^{\varepsilon}}=X^{1-\varepsilon}$. Hence we must have that
\[
X^{1-\varepsilon/2} \ll_{f,M,\varepsilon} \frac{1}{y} \sum_{\substack{X \le x \le 2MX \\ x \notin S}}\sum_{\substack{ x \le n \le x+y \\ }} C(f,n)
\ll_{f,M, \eps} \frac{1}{y} \cdot y X^{\eps} \sum_{\substack{X \le x \le 2MX \\ x \notin S}} 1,
\]
so that $\# \{ X \le x \le 2MX : x \notin S\} \gg_{f,M, \eps}  X^{1-\frac32 \varepsilon}$.
\end{proof}

\subsection{Proof of Proposition \ref{prop:cancellation}} \label{sec:cancellation}
Throughout this section we write
\begin{equation} \label{eq:Wdef}
W(u)=u^{\frac{k-1/2}{2}}e^{-2\pi u}.
\end{equation}
The proof of the proposition proceeds directly, beginning with an application of Cauchy-Schwarz. This leads naturally to a shifted convolution sum of Fourier coefficients of $f$ over squarefree integers and to bound this sum we require the following fairly standard result.

\begin{proposition} \label{prop:scp}
Let $f \in S^+_{k+\frac12}(4N)$.
Then for $1 \leq r\leq X^{\frac{1}{102}}$, $0 < |h| < X^{\frac{1}{2}}$, and $v\in \mathbb Z$ with $(v, r) = 1$ we have for any given $\varepsilon > 0$ that
  \begin{equation} \label{eq:scpest}
  \sum_{n \ge 1} c(f,n) c(f,n+h)  e \Big ( \frac{n v}{r} \Big )  W\left(\frac{n}{X} \right) W\left( \frac{n+h}{X}\right)  \ll_{f,\varepsilon} X^{1 - \frac{1}{102} + \varepsilon}.
  \end{equation}
\end{proposition}

\begin{proof}
This is an extension of Proposition 6.1 of \cite{LesterRadziwill2019signs} to the case of general level and below we will describe how to adapt the arguments given there to this case. The initial step is to use the Fourier expansion of $f$ to express the l.h.s. of \eqref{eq:scpest} as
\[ \tst
=\frac{1}{X^{k-\frac12}} \int_0^1 f\left(\alpha+\frac{v}{q}+\frac{i}{X} \right) f\left(-\alpha+\frac{i}{X} \right) e(-\alpha h) \, d\alpha .
\]
We now use the circle method following Jutila \cite{Jutila}, as in \cite[Proposition 2]{Harcos}. An important feature in Jutila's version of the circle method is that we have freedom over our choice of moduli, which we choose as follows
\[\mathcal Q=\{ Q \le q \le 2Q : q=4Nr p  \text{ and } p \equiv 1 \pmod 4 \text{  is prime}\}.\]
Write $R=\sum_{q \in \mathcal Q} \varphi(q)$.  Upon applying
\cite[Proposition 2]{Harcos} with $\delta=Q^{-2+\eta}$, $Q=X^{1/2+2\eta}$, $r \le X^{\eta/8}$, where $\eta>0$ is chosen later, we get that up to a term of size $O_f(X^{1-\eta/8+\varepsilon})$ the above integral equals
\begin{equation} \label{eq:circle}
 \frac{Q^{2-\eta}}{2  R X^{k - \frac 12}}\sum_{q \in \mathcal{Q}} \, \, \sumstar_{d \pamod q} e (-\tfrac{h d}{q}) \int_{-\delta}^{\delta} f  (\tfrac{d}{q}+\tfrac{v}{r}+\alpha+\tfrac{i}{X}) f(\tfrac{-d}{q}-\alpha+\tfrac{i}{X}) e (- \alpha h) d \alpha.
\end{equation}
Notably, to estimate the error term we use that $y^{\frac{k}{2} + \frac14} |f(z)|$ is bounded on $\mathbb H$, since $f$ is a cusp form.

Since we have chosen our moduli $q \in \mathcal Q$ so that $4N|q$ we are able to use the modularity of $f$ by applying \cite[Lemma 6.1]{LesterRadziwill2019signs}, which extends to general level in straightforward way, then once again use the Fourier expansion of $f$.
Consequently,
we have transformed the original sum on the l.h.s. of \eqref{eq:scpest}, which is effectively over $n\le X^{1+\varepsilon}$, to dual sums which are effectively over $m,n \le X^{\varepsilon} Q^2/X $. The summands in the dual sums include the Fourier coefficients of $f$ twisted by additive characters and factors from the half-integral weight multiplier system along with a Kloosterman sum $S(\star, -h;p)$, where the first argument, $\star$, depends on $N,p,m,n,v,r$.
An important observation is that since $p$ is a prime with $0< |h| < p$,  the Weil bound gives $|S(\star, -h;p)| \le 2\sqrt{p}$ for any $\star \in \mathbb Z$. Using the Weil bound and estimating the dual sums over $m,n$ by applying Cauchy-Schwarz and \eqref{eq:parseval} to handle the Fourier coefficients of $f$ we can show that \eqref{eq:circle} is bounded by
\[
\ll_{f, \eps} X^{3/4+49\eta/16+\varepsilon}.
\]
Recalling our earlier error term of $O(X^{1-\eta/8+\varepsilon})$, which arose from applying Jutila's circle method we now take $\eta=4/51$ to complete the proof.
\end{proof}

To sum over squarefree integers we will sieve out integers that have a square divisor and require the following estimate for sums of Fourier coefficients.

\begin{lemma} \label{lem:sumbd}
Let $f \in S_{k+\frac12}^+(4N)$. Then
\[
\sum_{\substack{n \le X \\  (n,2N)=1 \\ d^2 |n }} |c(f,n)| \ll_{f,\varepsilon} \frac{X^{1+\varepsilon}}{d^2}.
\]
\end{lemma}

\begin{proof}
Just as in the proof of Proposition \ref{prop:moments}, using
\eqref{e:heckebasis} it suffices to consider the case $f=f_1|U(r^2)$,
where $f_1 \in S_{k+\frac12}^{+,\text{new}}(4\ell)$ is a newform with $r\ell |N$. For $(n,2N)=1$ write $n=s^2m$ where $m$ is odd and squarefree. Also, let $t=rs$ and note $(t,\ell)=1$ since $N$ is squarefree.
Applying Lemma \ref{l:fourierUr2} we have that
\[ \tst
c(f_1|U(r^2),n)=c(f_1, t^2 m)  \ll_\eps t^{\varepsilon} |c(f_1,m)|.
\]
Using this bound then applying Cauchy-Schwarz and \eqref{eq:parseval} we get that
\[ \tst
\sum_{\substack{n \le X \\  (n,2N)=1 \\ d^2 |n }} |c(f_1|U(r^2),n)| \ll_\eps X^{\varepsilon} \sum_{\substack{t \le \sqrt{X} \\ d|t}} \sum_{ m \le \frac{X}{t^2} } |c(f_1,m)| \ll_{f, \eps} X^{\varepsilon} \sum_{\substack{t \le \sqrt{X} \\ d|t}} \frac{X}{t^2} \ll_\eps \frac{X^{1+\varepsilon}}{d^2}.
\]
\end{proof}

\begin{proof}[Proof of Proposition \ref{prop:cancellation}]


To handle the condition that $2n$ is squarefree we first recall that the indicator function of squarefree numbers is $\mu^2(n)=\sum_{d^2|n} \mu(d)$. We then treat the cases of divisors $d\le Y$ and $d >Y$ separately, and let
$$
\mu^2_{\leq Y}(n) = \sum_{\substack{d^2 |n \\ d \leq Y }} \mu(d)  \ , \qquad \ \mu^2_{> Y}(n) = \sum_{\substack{d^2 |n \\ d > Y, }} \mu(d).
$$
First we consider the large divisors and get that
\begin{align*}
  \sum_{X \leq x \leq 2X} \bigg|\sum_{\substack{x \leq n \leq x + y \\ (n,2N)=1}} c(f,n) \mu^2_{>Y }(n) \bigg| \leq y \sum_{\substack{n \leq 4 X \\ (n,2N)=1}} |c(f,n) \mu^2_{> Y}(n)|.
\end{align*}
Using the definition of $\mu^2_{> Y}(n)$ and applying Lemma \ref{lem:sumbd} we see that the r.h.s. above is
\begin{equation} \label{eq:largedbd}
\ll y \sum_{\substack{d > Y }} \sum_{\substack{n \leq 4 X \\ d^2 | n}} |c(f,n) |  \ll_{f, \eps} y X^{1+\varepsilon} \sum_{d >Y} \frac{1}{d^2} \ll y \frac{X^{1+\varepsilon}}{Y} .
\end{equation}
For $ Y \ge \sqrt{y}X^{\varepsilon}$ this is $\ll X \sqrt{y}$, as needed.

 Next, we will consider the contribution from the small divisors $d \le Y$. Let \linebreak $\widetilde C(f,n)=c(f,n)\mu_{\le Y}^2(n)1_{(n,2N)=1}$. Applying Cauchy-Schwarz and using that $W(u)^2 \gg 1$, for any $u \in [1,2]$ we get
\[\tst
\begin{split}
\sum_{X \le x \le 2 X} \Big | \sum_{x \le n \le x+y} \widetilde C(f,n)  \Big|
\le& \sqrt{X} \Big( \sum_{X \le x \le 2X} \Big | \sum_{x \le n \le x + y} \widetilde C(f,n) \Big |^2\Big)^{1/2} \\
\ll& \sqrt{X}  \Big( \sum_{x \ge 1} \Big| \sum_{x \le n \le x + y } \widetilde C(f,n)  \Big |^2 W\left( \frac{x}{X}\right)^2\Big)^{1/2}.
\end{split}
\]
Assume $y \le X^{1/4}$. We use the convention that $c(f,n)=0$ if $n \notin \mathbb N$.
To estimate the inner sums on the r.h.s. we expand the square, combine appropriate terms, use that $W$ is a smooth function, and apply \eqref{eq:parseval} to get that the r.h.s. above equals
\begin{equation} \label{eq:expand}
\begin{split}
&=\sum_{0 \leq h_1, h_2 \leq y} \sum_{\substack{n \ge 1 }} \widetilde C(f,n+h_1)\widetilde C(f,n+h_2)  W\left(\frac{n}{X} \right)^2 \\
  &=
  \sum_{\substack{|h| \leq y}} \sum_{\substack{n \ge 1}} \widetilde C(f,n)
  \widetilde C(f,n+h)  W\Big (\frac{n} {X} \Big) W\Big ( \frac{n + h}{ X} \Big )
  \sum_{\substack{0 \le h_1,h_2 \le y \\ h_2-h_1=h}} \left(1+O\left( \frac{y}{n}+\frac{y}{X} \right) \right) + O_f(1) \\
  &=
  \sum_{|h| \leq y} (y + 1 - |h|) \sum_{\substack{n \ge 1 }} \widetilde C(f,n)
  \widetilde C(f,n + h) W \Big (\frac{n}{X} \Big) W \Big ( \frac{n + h}{X}  \Big )
  +O_{f, \eps}\left(y^3  X^{\varepsilon} \right),
\end{split}
\end{equation}
Using \eqref{eq:parseval} once again we get that the term with $h=0$ in the sum on the r.h.s. above contributes
\begin{equation} \label{eq:diagonal}
  \ll
  y \sum_{n \ge 1} |c(f,n) |^2 W \left ( \frac{n}{X} \right )^2 \ll_{f, \eps} yX.
\end{equation}

We will next estimate the contribution from the terms in \eqref{eq:expand} with $h \neq 0$.
Recalling the definition of $\mu^2_{\leq Y}$ and using that
\begin{equation} \notag
1_{(n,2N)=1}=\sum_{\substack{d|(n,2N)}} \mu(d)=\sum_{\substack{d|2N \\ d|n}} \mu(d),
\end{equation}
it follows that the contribution to the r.h.s. of \eqref{eq:expand}  from the terms with $h \neq 0$ is
\begin{equation} \label{eq:trivialbd}
  \ll y   \sum_{\substack{0 < |h| \leq y \\ d_1, d_2 \leq Y  \\ d_3,d_4 | 2N}}  \bigg|
  \sum_{\substack{d_1^2| n, d_2^2 | n + h \\ d_3|n, d_4|n+h} } c(f,n)
  c(f,n+h)   W\Big ( \frac{n}{X} \Big ) W \Big ( \frac{n + h}{X} \Big )  \bigg|.
\end{equation}
For $n$ with $d_1^2 | n$,  $d_2^2 | n + h$, $d_3|n$, $d_4|n+h$ we have $n \equiv a \, (\tmop{mod} \, r)$ for some $a,r \in \mathbb Z$ with $ r\le 16 N^2 Y^4$. Using additive characters to detect this congruence, we get that the inner sum above is
\begin{equation} \notag
\begin{split}
  & =\sum_{n \equiv a \pamod{r}} c(f,n)
  c(f,n+h) W \Big ( \frac{n}{X} \Big ) W \Big ( \frac{n + h}{X} \Big ) \\
 &=  \frac{1}{r} \sum_{0 \le v < r} e\left(\frac{-av}{r} \right) \sum_{n \ge 1} c(f,n)
c(f,n+h)e \left ( \frac{n v}{r} \right ) W \Big ( \frac{n}{X} \Big ) W \Big ( \frac{n + h}{X} \Big ).
  \end{split}
\end{equation}
For $v \neq 0$ write $v/r=v'/r'$ with $(r',v')=1$ and if $v=0$ we set $r'=1$. Applying Proposition \ref{prop:scp},
we get that the sum above is
$\ll_{f, \eps} X^{1-\frac{1}{102}+\varepsilon}$ provided that $r' \le X^{\frac{1}{102}}$ and $0<|h| < X^{\frac{1}{2}}$. Hence, by this along with \eqref{eq:diagonal} we conclude that for $16N^2 Y^4 \le X^{\frac{1}{102}}$ that the r.h.s. of \eqref{eq:expand} is
$
\ll_{f, \eps} y X+ y^2 Y^2 X^{1-\frac{1}{102}+\varepsilon}+y^3 X^{\varepsilon},
$
which is $\ll y X$, as needed, provided that $y \le X^{1/4}$ and \begin{equation}\label{eq:constraint}yY^2X^{\varepsilon} \ll X^{\frac{1}{102}}.\end{equation}

It remains to optimize our parameters. Recall that to handle the contribution of the small divisors we required $Y \ge \sqrt{y} X^{\varepsilon}$. We now choose $Y=\sqrt{y} X^{\varepsilon}$. Taking the constraint \eqref{eq:constraint} into account the largest we can choose $y$ is $y=X^{\frac{1}{204}-\frac{3}{2}\varepsilon}$. We conclude by noting that with these choices we have $16 N^2 Y^4 \ll_N X^{\frac{1}{102}-3\varepsilon} \le X^{\frac{1}{102}}$ as required for the application of Proposition \ref{prop:scp}. \end{proof}

\section{Large values for coefficients of half-integral weight forms}\label{s:large}
The main result of this section, Theorem \ref{thm:large2} below,  generalizes  \cite[Theorem 1]{gun2020large} (which treated the case $N=1$).
\subsection{Statement of main result}

\begin{theorem} \label{thm:large2}
Let $k \ge 2$ be an integer, $N \ge 1$ be odd and squarefree, and $h \in S_{k+\frac12}^+(4N)$ be a cusp form. Let $\varepsilon>0$ be fixed. Then for all $X$ sufficiently large, there exist at least $X^{1-\varepsilon}$ odd squarefree integers $n$ coprime to $N$ such that $X \le n \le 2X$ and
 \[\tst |c(h,n)| \ge  \exp\left(\frac{1}{82} \sqrt{\frac{\log n}{\log \log n}} \right). \]
\end{theorem}

\noindent We will prove this theorem by combining methods of \cite{gun2020large} and \cite{RadziwillSound, LesterRadziwill2019signs}. Our first job is to reduce the question to bounding central values of $L$-functions. This is done by using the explicit form of Waldspurger's formula due to Kohnen (Proposition \ref{p:walds} above).

\subsection{Reduction to bounds on $L$-values}


Fix an integer $k \ge 2$ and an odd squarefree integer $N\geq 1$ throughout Section $4$. Let $h$ be as in Theorem \ref{thm:large2}. We use the basis \eqref{e:heckebasis} to write \[h = \sum_{\substack{r,\ell\geq 1\\
r\ell|N}}\sum_{f\in\B_{k+\frac12,4\ell}^\text{new}} \alpha_{r,\ell, f} f|U(r^2),\] where the coefficients $\alpha_{r, \ell, f}$ depend only on $r, \ell, f,$ and $h$. For each odd squarefree $n$, we use Lemma \ref{l:fourierUr2} to get the following identity for the Fourier coefficients:
\[
c(h,n)= \sum_{\ell|N}\sum_{f\in\B_{k+\frac12,4\ell}^\text{new}} c(f,n) \sum_{r| \frac{N}{\ell}}\alpha_{r,\ell,f} \prod_{p|r}\left(\lambda_f(p) - \frac{1}{\sqrt{p}}\left(\frac{(-1)^kn}{p} \right) \right).
\]

We already know that $c(h,n) \neq 0$ for some odd squarefree $n$ (this follows from \cite{sahaschmidt} for example). So there exist $\ell_0|N$, $f_0\in\B_{k+\frac12,4\ell_0}^\new$, and a reduced residue class $\eta$ mod $4N$ such that
\[\tst \eta \equiv 1 \pmod{4}, \quad \left(\frac{\eta}p\right)=w_p \quad \text{ for each } p|\ell_0,\]
\[\sum_{r| \frac{N}{\ell_0}}\alpha_{r,\ell_0,f_0} \prod_{p|r}\left(\lambda_{f_0}(p) - \frac{1}{\sqrt{p}}\left(\frac{\eta}{p} \right) \right) \neq 0. \]
Above, $w_p$ is the eigenvalue of the Atkin-Lehner operator $W_p$ acting on $f_0$.
For brevity, we denote for each $f \in \B_{k+\frac12,4\ell}^\new$
\[\beta_{f} = \sum_{r| \frac{N}{\ell}}\alpha_{r,\ell,f} \prod_{p|r}\left(\lambda_{f}(p) - \frac{1}{\sqrt{p}}\left(\frac{\eta}{p} \right) \right).\]
We will denote the Shimura lift of $f\in\B_{k+\frac12,4\ell}^\new$ by $g_{f}\in S_{2k}^\new(\ell)$ with Fourier coefficients $m^{k - \frac12}\lambda_{g_f}(m)$ normalized so that $\lambda_{g_f}(1)=1$. Also, write $g_0$ for $g_{f_0}$ and $m^{k-\frac12}\lambda_0(m)$ for its Fourier coefficients. For each odd squarefree integer $n$ such that $d=(-1)^k n \equiv \eta$ (mod $4N$), we use the triangle inequality, Cauchy-Schwarz, and Proposition \ref{p:walds} to obtain
\begin{equation}\label{e:cgestimate}\begin{split}\tst |c(h, n)| &= \tst \left|\sum_{\ell|N}\sum_{f\in\B_{k+\frac12,4\ell}^\text{new}}
 c(f,n) \beta_f\right| \\  & \tst \ge   |\beta_{f_0} c_{f_0}(n)|    -\sqrt{\left(\sum_{\ell|N}\sum_{f_0 \ne f\in\B_{k+\frac12,4\ell}^\text{new}} | c(f,n) |^2 \right) \left(\sum_{\ell|N}\sum_{f_0 \ne f\in\B_{k+\frac12,4\ell}^\text{new}} | \beta_f|^2 \right)}\\  & \tst \ge A \sqrt{L\left(\frac12,g_0\otimes\chi_d\right)} - B \sqrt{\sum_{\ell|N}\sum_{f_0 \ne f\in\B_{k+\frac12,4\ell}^\text{new}} L\left(\frac12,g_f\otimes\chi_d\right)},\end{split}
 \end{equation}
where $A>0$ and $B>0$ are independent of $d$.
Now, Theorem \ref{thm:large2} follows from the following auxiliary result.
\begin{proposition}\label{p:largep1} Let $C\ge 0$ be a constant, $\varepsilon>0$, and let $\eta$ (mod $4N$) be a fixed reduced residue class with $\eta\equiv 1$ (mod $4$). Given $\ell_0|N$ let $f_0 \in \B_{k+\frac12,4\ell_0}^\new$ be a newform as above with Shimura lift $g_0$. For sufficiently large $X$, there are $\ge X^{1-\varepsilon}$ odd squarefree integers $n \in [X, 2X]$, such that $d=(-1)^kn\equiv\eta\,(\text{mod }4N)$ and
\begin{align}\label{main-estimate}
L\left(\tfrac12,g_0\otimes\chi_d\right)>C\sum_{\ell|N}\sum_{\substack{f\in\B_{k+\frac12,4\ell}^\new\\
(\ell,f)\neq(\ell_0,f_0)}}L\left(\tfrac12,g_f\otimes\chi_d\right)+ \exp\left(\frac1{40}\sqrt{\frac{\log X}{\log\log X}}\right).
\end{align}
\end{proposition}
We first explain how Proposition \ref{p:largep1} implies Theorem \ref{thm:large2}. 
Put $L_0 = L\left(\frac12,g_0\otimes\chi_d\right)$ and $L_1 = \sum_{\ell|N}\sum_{f_0 \ne f\in\B_{k+\frac12,4\ell}^\new}L\left(\frac12,g_f\otimes\chi_d\right)$. Put $C=2B^2/A^2$. Combining  Proposition \ref{p:largep1}  and \eqref{e:cgestimate},
\begin{align*}\tst
|c(h, n)| &\ge A \sqrt{L_0} - B \sqrt{L_1} \ > A \sqrt{CL_1 + \exp\left(\frac1{40}\sqrt{\frac{\log X}{\log\log X}}\right)} -  B \sqrt{L_1} \\ & \ge \frac{A}{\sqrt{2}}\left(\sqrt{CL_1} + \exp\left(\frac1{80}\sqrt{\frac{\log X}{\log\log X}}\right)\right)  -  B \sqrt{L_1}\\ &=  \frac{A}{\sqrt{2}} \exp\left(\frac1{80}\sqrt{\frac{\log X}{\log\log X}}\right) \ge  \exp\left(\frac{1}{82} \sqrt{\frac{\log n}{\log \log n}} \right)
\end{align*}
for sufficiently large $X$. We now proceed with the proof of Proposition \ref{p:largep1}. For $\eta\equiv 1$ (mod $4$), let
\[\mathcal D_{N,\eta}=\{d=(-1)^kn\,:\,\mu^2(2n)\neq 0,\,(n,N)=1,\, d\equiv \eta\,(\text{mod }4N)\} \]
and
\[\mathcal D_{N,\eta}(X)=\mathcal D_{N,\eta}\cap[X,2X].\]
For each such $d$ we introduce a resonance polynomial
\begin{align*}
R(d)=\sum_{m\leq M}r(m)\lambda_0(m)\chi_d(m),
\end{align*}
where the coefficients $r(m)$ are multiplicative and supported on squarefree integers. At primes we set
\begin{align*}
r(p)=\begin{cases}
\frac L{\sqrt p \log p}&\quad\text{if }L^2\leq p\leq L^4\,\text{and }p\nmid N \\
0&\quad\text{otherwise}
\end{cases}
\end{align*}
where $L=\frac18\sqrt{\log M\log\log M}$ with $M=X^{1/24}$.

The strategy to prove (\ref{main-estimate}) is to consider the quantity
\begin{align*}
\sum_{d\in\mathcal D_{N,\eta}(X)}L\left(\tfrac12,g_0\otimes\chi_d\right)|R(d)|^2.
\end{align*}
Let $\mathcal S$ be the subset of $\mathcal D_{N,\eta}(X)$ for which the estimate (\ref{main-estimate}) holds. Then certainly
\begin{align*}
&\tst \sum_{d\in\mathcal D_{N,\eta}(X)}L\left(\frac12,g_0\otimes\chi_d\right)|R(d)|^2\\
&\tst \leq\sum_{d\in\mathcal D_{N,\eta}(X)}\left(\sum_{\substack{\ell=1\\
\ell|N}}^\infty\sum_{\substack{f\in\B_{k+\frac12,4\ell}^\new\\
(\ell,f)\neq(\ell_0,f_0)}}CL\left(\frac12,g_f\otimes\chi_d\right)+\exp\left(\frac1{40}\sqrt{\frac{\log X}{\log\log X}}\right)\right)|R(d)|^2\\
&\tst\qquad\qquad+\sum_{d\in\mathcal S}L\left(\frac12,g_0\otimes\chi_d\right)|R(d)|^2.
\end{align*}

Suppose that the following estimates hold:

\begin{align}\label{estimate1}
\sum_{d\in\mathcal D_{N,\eta}(X)}L\left(\tfrac12,g_0\otimes\chi_d\right)|R(d)|^2\gg_{k,N} X\cdot\mathcal R\cdot\exp\left(\left(\frac12+o(1)\right)\frac L{\log L}\right),
\end{align}

\begin{align}\label{estimate2}
\sum_{d\in\mathcal D_{N,\eta}(X)}L\left(\tfrac12,g_f\otimes\chi_d\right)|R(d)|^2\ll_{k,N} X\cdot\mathcal R\cdot\exp\left(o\left(\frac L{\log L}\right)\right) \quad \text{for $f\neq f_0$,}
\end{align}

\begin{align}\label{estimate3}
\sum_{d\in\mathcal D_{N,\eta}(X)}|R(d)|^2\leq\frac{2X}{\pi^2}\cdot\mathcal R+O(X),
\end{align}

\begin{align}\label{estimate4}
\sum_{d\in\mathcal D_{N,\eta}(X)}|R(d)|^6\ll X\cdot\exp\left(O\left(\frac{\log X}{\log\log X}\right)\right),
\end{align}
where
\begin{align*}
\mathcal R=\prod_{L^2\leq p\leq L^4}\left(1+r(p)^2\lambda_0(p)^2\right).
\end{align*}

Assuming \eqref{estimate1}-\eqref{estimate4} the proof of Proposition \ref{p:largep1} can be finished as follows. We observe that
\begin{align*}
X\cdot\mathcal R\cdot\exp\left(\left(\frac12+o(1)\right)\frac L{\log L}\right)&\ll_{k,N}\sum_{d\in\mathcal D_{N,\eta}(X)}L\left(\tfrac12,g_0\otimes\chi_d\right)|R(d)|^2\\
&\ll_{k,N} X\cdot\mathcal R\cdot\exp\left(o\left(\frac L{\log L}\right)\right)+\sum_{d\in\mathcal S}L\left(\tfrac12,g_0\otimes\chi_d\right)|R(d)|^2
\end{align*}
by using \eqref{estimate1}, \eqref{estimate2} and \eqref{estimate3}. Hence,
\begin{align}\label{lower-bound}
X\cdot\mathcal R\cdot\exp\left(\left(\frac12+o(1)\right)\frac L{\log L}\right)\ll_{k,N} \sum_{d\in\mathcal S}L\left(\tfrac12,g_0\otimes\chi_d\right)|R(d)|^2.
\end{align}
On the other hand, r.h.s of the previous display can be estimated by H\"older's inequality and \eqref{estimate4} as
\begin{align*}
&\leq \left(\sum_{d\in\mathcal D_{N,\eta}(X)}L\left(\tfrac12,g_0\otimes\chi_d\right)^2\right)^{1/2}\cdot |\mathcal S|^{1/6}\left(\sum_{d\in\mathcal S}|R(d)|^6\right)^{1/3}\\
&\ll_{k,N} X^{1/2+\varepsilon/2}|\mathcal S|^{1/6}\cdot\left(X\cdot\exp\left(O\left(\frac{\log X}{\log\log X}\right)\right)\right)^{1/3},
\end{align*}
where, as before, the average of the squares of central $L$-values is estimated by using the quadratic large sieve of Heath-Brown \cite{HB} (here we can extend the sum to all fundamental discriminants $\le X$ in magnitude by non-negativity). Combining this with \eqref{lower-bound} gives $|\mathcal S|\gg_{k,N} X^{1-\varepsilon/2} \ge X^{1-\eps}$, as desired.

So it suffices to establish estimates \eqref{estimate1}-\eqref{estimate4}. Notice that \eqref{estimate3} and \eqref{estimate4} follow directly from Proposition 3 of \cite{gun2020large} by simply estimating
\[\sum_{d\in\mathcal D_{N,\eta}(X)}|R(d)|^2\leq\sum_{\substack{X\leq (-1)^kd\leq 2X\\
d\equiv 1\,(\text{mod }4)}}|R(d)|^2\qquad\text{and }\sum_{d\in\mathcal D_{N,\eta}(X)}|R(d)|^6\leq\sum_{\substack{X\leq (-1)^kd\leq 2X\\
d\equiv 1\,(\text{mod }4)}}|R(d)|^6.\]
The other two estimates are consequences of the following first moment result.
\begin{proposition}\label{key-proposition}
Let $N\geq 1$ be a positive integer and $g\in S_{2k}^\new(\ell)$ for some $\ell|N$. Suppose that $u$ is an odd positive integer coprime to $N$ and write $u=u_1u_2^2$ with $u_1$ squarefree. Let $\Phi$ denote a smooth and compactly supported function in $[1/2,5/2]$. Then
\begin{align*}
&\sum_{d\in \mathcal D_{N,\eta}}L\left(\tfrac12,g\otimes\chi_d\right)\chi_d(u)\Phi\left(\frac{|d|}X\right)\\
&=\frac{X\lambda_g(u_1)}{2Nu_1^{1/2}}\left(\int_0^\infty\Phi(\xi)\,\mathrm d \xi\right)L_{g,\eta}\left(\tfrac12\right)L\left(1,\mathrm{Sym}^2g\right)\mathcal G(1;u)+O_{k,N, \Phi, \varepsilon}\left(X^{7/8+\varepsilon}u^{3/8+\varepsilon}\right).
\end{align*}
Here $L_{g,\eta}(1/2)\neq 0$ is the Dirichlet series given in (\ref{DS1}) and $\mathcal G(1;u)$ is the Euler product defined in (\ref{firstEulerproduct}). Furthermore, $\mathcal G(1;\cdot)$ is a multiplicative function satisfying $\mathcal G(1;p^k)=1+O(1/p)$ at prime powers.
\end{proposition}

\subsection{A twisted first moment asymptotic}

In this subsection, we sketch the proof of the proposition above. The starting point is the approximate functional equation (which follows by an easy modification of the proof of Lemma 5 in \cite{RadziwillSound}) saying that
\begin{align*}
 L\left(\frac12,g\otimes\chi_d\right)=2\sum_{\substack{m=1\\
(m,2N)=1}}^\infty\frac{\lambda_g(m)\chi_d(m)}{\sqrt m}W_{g,\eta}\left(\frac m{|d|}\right),
\end{align*}
where $W_{g,\eta}$ is a smooth weight function defined by, for $c>1/2$,
\begin{align}\label{DS1}
W_{g,\eta}(\xi)=\frac1{2\pi i}\int_{(c)}L_{g,\eta}\left(s+\frac12\right)\frac{\Gamma(s+k)}{\Gamma(k)}\left(\frac{\sqrt \ell}{2\pi\xi}\right)^s\frac{\mathrm d s}s, \qquad L_{g,\eta}(s)=\sum_{\substack{m=1\\
p|m\implies p|2N}}^\infty\frac{\lambda_g(m)\chi_d(m)}{m^s}.
\end{align}
Notice that the value of $L_{g,\eta}(s)$ is the same for each $d\equiv\eta$ (mod $4N$). The weight function satisfies $W_{g,\eta}(\xi)=L_{g,\eta}(1/2)+O_{k,N,\varepsilon}(\xi^{\frac12-\varepsilon})$ as $\xi\longrightarrow 0$ and $W_{g,\eta}(\xi)\ll_{k,N,A}|\xi|^{-A}$ for any $A\geq 1$ as $\xi\longrightarrow\infty$. Thus the sum we have to evaluate takes the shape
\begin{align*}
2\sum_{\substack{m=1\\
(m,2N)=1}}^\infty\frac{\lambda_g(m)}{\sqrt m}\sum_{d\in\mathcal D_{N,\eta}}\chi_d(mu)W_{g,\eta}\left(\frac m{|d|}\right)\Phi\left(\frac{|d|}X\right).
\end{align*}
Notice that by definition any $d\in\mathcal D_{N,\eta}$ is squarefree and odd. We pick out this property by the identity
\begin{equation}
\sum_{\substack{\alpha=1\\
(\alpha,2N)=1\\
\alpha^2|d}}^\infty\mu(\alpha)=\begin{cases}
1\quad\text{if }d\,\text{ is squarefree}\\
0\quad\text{otherwise}
\end{cases}
\end{equation}
Note that the above identity holds without the condition $(\alpha,2N)=1$, but this can be added as by construction $(d,2N)=1$ for $d\in\mathcal D_{N,\eta}$. Inserting this to the above expression gives that the $d$-sum is given by
\begin{align*}
\sum_{\substack{\alpha=1\\
(\alpha,2Nu)=1}}^\infty\mu(\alpha)\left(\frac{\alpha^2}{mu}\right)\sum_{r\equiv\eta\overline{\alpha^2}\,(\text{mod }4N)}\left(\frac r{mu}\right)W_{g,\eta}\left(\frac m{r\alpha^2}\right)\Phi\left(\frac{r\alpha^2}X\right).
\end{align*}
We will evaluate the $r$-sum by applying a version of Poisson summation formula \cite[Lemma 7]{RadziwillSound}. The terms where $mu$ is a square will contribute the main term in the zero-frequency term on the dual side and the rest will give the error term. By using standard arguments \cite{LesterRadziwill2019signs} the contribution of the latter terms can be bounded by $\ll_{k,N, \Phi, \varepsilon} X^{7/8+\varepsilon}u^{3/8+\varepsilon}$.

Using (\ref{DS1}) the zero-frequency contribution is given by
\begin{align}\label{roughmainterm}
\frac{X}{2N}\cdot\frac1{2\pi i}\int_{(c)}\left(\int_0^\infty\Phi(\xi)\xi^s\,\mathrm d \xi\right)L_{g,\eta}\left(s+\tfrac12\right)\frac{\Gamma(s+k)}{\Gamma(k)}\left(\frac{\sqrt \ell X}{2\pi}\right)^s
\sum_{\substack{\alpha=1\\
(\alpha,2Nu)=1}}^\infty\frac{\mu(\alpha)}{\alpha^2}\sum_{\substack{m=1\\
(m,2N\alpha)=1\\
mu \text{ is a square}}}^\infty\frac{\lambda_g(m)}{m^{s+1/2}}\frac{\varphi(mu)}{mu}\,\frac{\mathrm d s}s.
\end{align}
A simple computation shows that, for $(m,2N)=1$,
\begin{align*}
\sum_{\substack{\alpha=1\\
(\alpha,2Num)=1}}^\infty\frac{\mu(\alpha)}{\alpha^2}\cdot\frac{\varphi(mu)}{mu}=\prod_{p\nmid 2N}\left(1-\frac1{p^2}\right)\prod_{p|um}\left(1+\frac1p\right)^{-1}, \quad\text{ leading to}
\end{align*}

\begin{align}\label{productform}
\sum_{\substack{\alpha=1\\
(\alpha,2Nu)=1}}^\infty\frac{\mu(\alpha)}{\alpha^2}\sum_{\substack{m=1\\
(m,2N\alpha)=1\\
mu \text{ is a square}}}^\infty\frac{\lambda_g(m)}{m^{s+1/2}}\cdot\frac{\varphi(mu)}{mu}=\prod_{p\nmid 2N}\left(1-\frac1{p^2}\right)\sum_{\substack{m=1\\
(m,2N)=1\\
mu \text{ is a square}}}^\infty\frac{\lambda_g(m)}{m^{s+1/2}}\prod_{p|um}\left(1+\frac1p\right)^{-1}.
\end{align}
For $mu\text{ a square}$ and $u=u_1u_2^2$ with $u_1$ squarefree, it follows that $m=u_1\ell^2$ for some $\ell\in\mathbb Z$. Hence the r.h.s of the previous display is
\begin{align*}
&\prod_{p\nmid 2N}\left(1-\frac1{p^2}\right)\sum_{\substack{\ell=1\\
(\ell,2N)=1}}^\infty\frac{\lambda_g(u_1\ell^2)}{u_1^{s+1/2}\ell^{2s+1}}\prod_{p\mid u\ell}\left(1+\frac1p\right)^{-1}\\
&=\frac1{u_1^{s+1/2}}\prod_{p\nmid 2N}\left(1-\frac1{p^2}\right)\prod_{p\nmid 2Nu}\left(1+\sum_{k=1}^\infty\frac{\lambda_g(p^{2k})}{p^{k(2s+1)}}\left(1+\frac1p\right)^{-1}\right)\prod_{p|u}\left(\sum_{k=0}^\infty\frac{\lambda_g(p^{2k+\text{ord}_p(u_1)})}{p^{k(2s+1)}}\right)\left(1+\frac1p\right)^{-1}\\
&=\frac1{u_1^{s+1/2}}L(2s+1,\text{Sym}^2g)\prod_{p|2N}L_p(2s+1,\text{Sym}^2h)^{-1}\prod_{p|u}\left(1-\frac{1}p\right)\left(\sum_{k=0}^\infty\frac{\lambda_g(p^{2k+\text{ord}_p(u_1)})}{p^{k(2s+1)}}\right)\\
&\quad\cdot L_p(2s+1,\text{Sym}^2g)^{-1}\prod_{p\nmid 2Nu}\left(1-\frac1{p^2}\right)\left(1+\sum_{k=1}^\infty\frac{\lambda_g(p^{2k})}{p^{k(2s+1)}}\left(1+\frac1p\right)^{-1}\right)L_p(2s+1,\text{Sym}^2g)^{-1}.
\end{align*}
From this it is easy to see, by using the Euler product expression of the symmetric square $L$-function, that for $p\nmid 2Nu$ the corresponding Euler factor is
\[
=\left(1-\frac1p\right)\left(1-\frac1{p^{2s+1}}\right)\left(\frac1p\left(1-\frac{\alpha_p^2}{p^{2s+1}}\right)\left(1-\frac{\beta_p^2}{p^{2s+1}}\right)+1+\frac1{p^{2s+1}}\right),
\]
where $\{\alpha_p,\beta_p\}$ are the Satake parameters of the cusp form $g$ at $p$.

Similarly, for $p|u$, $p\nmid u_1$, the corresponding Euler factor is $\left(1-\frac{1}{p}\right)\left(1-\frac1{p^{4s+2}}\right)$.
For $p|u_1$, the corresponding Euler factor is $\left(1-\frac{1}{p}\right)\left(1-\frac1{p^{2s+1}}\right)\lambda_g(p)$
by using the relations $\lambda_g(p^j)=(\alpha_p^{j+1}-\beta_p^{j+1})/(\alpha_p-\beta_p)$ for $j\geq 0$ and $\alpha_p\beta_p=1$. And finally, for $p|2N$ the corresponding Euler factor is clearly $L_p(2s+1,\text{Sym}^2g)^{-1}$.

To summarize, the r.h.s of (\ref{productform}) equals
\[\frac{\lambda_g(u_1)}{u_1^{s+1/2}}L\left(2s+1,\text{Sym}^2g\right)\mathcal G(2s+1;u),\]
where $\mathcal G(2s+1;u)=\prod_p \mathcal G_p(2s+1;u)$ is the Euler product locally given by
\begin{align}\label{firstEulerproduct}
\mathcal G_p(2s+1;u)=\begin{cases}
L_p\left(2s+1,\text{Sym}^2 g\right)^{-1}&\quad\text{if }p|2N\\
\left(1-\frac1{p}\right)\left(1-\frac1{p^{2s+1}}\right)&\quad\text{if }p|u_1\\
\left(1-\frac1{p}\right)\left(1-\frac1{p^{4s+2}}\right)&\quad\text{if }p|u_2,\,p\nmid u_1\\
\left(1-\frac1p\right)\left(1-\frac1{p^{2s+1}}\right)\left(1+\frac1p\left(1-\frac{\alpha_p^2}{p^{2s+1}}\right)\left(1-\frac{\beta_p^2}{p^{2s+1}}\right)+\frac1{p^{2s+1}}\right)&\quad\text{if }p\nmid 2Nu.
\end{cases}
\end{align}
By estimating trivially it follows that $\mathcal G(2s+1;u)$ extends analytically to the domain $\Re(s)>-1/4$ and is bounded there by
\[ \ll\prod_{p\nmid 2Nu}\left(1+\frac{O(1)}{\sqrt p}\right)\ll_{N,\varepsilon} u^\varepsilon.\]
Consequently the $s$-integrand in (\ref{roughmainterm}) extends to an analytic function in the above domain (apart from a simple pole at $s=0$). Thus moving the line of integration in \eqref{roughmainterm} to the line $\Re(s)=-1/4+\varepsilon$ shows that the expression equals
\[ \frac{X\lambda_g(u_1)}{2Nu_1^{1/2}}\left(\int_0^\infty\Phi(\xi)\,\mathrm d \xi\right)L_{g,\eta}\left(\frac12\right)L\left(1,\text{Sym}^2g\right)\mathcal G(1;u)+O_{k,N, \Phi, \varepsilon}\left(X^{3/4+\varepsilon}\right),\]
where the main terms comes from the residue at $s=0$ and the error term from the contour shift. It follows immediately from the definition of $\mathcal G(s;u)$ that $\mathcal G(1;\cdot)$ is multiplicative and that $\mathcal G(1;p^k)=1+O(1/p)$ at prime powers. This concludes the sketch of the proof of Proposition \ref{key-proposition}.

\subsection{Proofs of the estimates} We are now ready to prove the estimates (\ref{estimate1}) and (\ref{estimate2}). As the arguments are similar to \cite{gun2020large}, we will be brief.  Let us denote
\[ \tst A_{g,\eta}(\Phi)=\frac1{2N}\left(\int_0^\infty\Phi(\xi)\,\mathrm d \xi\right)L_{g,\eta}\left(\frac12\right)L\left(1,\text{Sym}^2g\right)\]
and
\[\tst B_g(u)=\frac{\lambda_g(u_1)}{u_1^{1/2}}\mathcal G(1;u)\]
for $u=u_1u_2^2$ with $u_1$ squarefree.

Let $\Phi$ be a compactly supported smooth weight function. Our aim is now to evaluate
\[\sum_{d\in\mathcal D_{N,\eta}}L\left(\tfrac12,g\otimes\chi_d\right)|R(d)|^2\Phi\left(\frac{|d|}X\right)\]
for $g=g_f\in S_{2k}^\new(\ell)$, where $f\in\B_{k+\frac12,4\ell}^\new$ with $\ell|N$, by choosing $\Phi$ appropriately for given $f$.

By opening the definition of $R(d)$ and using Proposition \ref{key-proposition} the above sum equals
\begin{align*}
&\sum_{d\in\mathcal D_{N,\eta}}L\left(\tfrac12,g\otimes\chi_d\right)\Phi\left(\frac{|d|}X\right)\sum_{n_1,n_2\leq M}r(n_1)r(n_2)\lambda_0(n_1)\lambda_0(n_2)\chi_d(n_1n_2)\\
&=X\cdot A_{g,\eta}(\Phi)\sum_{n_1,n_2\leq M}r(n_1)r(n_2)\lambda_0(n_1)\lambda_0(n_2)B_g\left(n_1n_2\right)\\
&\qquad\qquad+O_{k,N,\Phi,\varepsilon}\left(X^{7/8+\varepsilon}\sum_{n_1,n_2\leq M}r(n_1)r(n_2)|\lambda_0(n_1)\lambda_0(n_2)|(n_1n_2)^{3/8+\varepsilon}\right)
\end{align*}
as $r(n)$ vanishes unless $(n,N)=1$. By using Deligne's bound we obtain as in \cite[\S 6]{gun2020large} that  the error term is $\ll_{k,N,\Phi} X^{99/100}$.


By making use of the fact that
\[\tst n_1n_2=\frac{n_1n_2}{(n_1,n_2)^2}\cdot(n_1,n_2)^2,\]
with the first factor on the r.h.s squarefree in our case as the function $r(n)$ is supported only on squarefree integers, the main term can be written as
\[\tst X\cdot A_{g,\eta}(\Phi)\sum_{n_1,n_2\leq M}r(n_1)r(n_2)\lambda_0(n_1)\lambda_0(n_2)\frac{\lambda_g\left(\frac{n_1n_2}{(n_1,n_2)^2}\right)}{\left(\frac{n_1n_2}{(n_1,n_2)^2}\right)^{1/2}}\cdot\mathcal G(1;n_1n_2).\]
Our aim is to use multiplicativity and so we need to  extend the sum over all integers. To do so we must show that the terms with max($n_1,n_2)>M$ can be added with a tolerable error. By using Rankin's trick these terms contribute
\begin{align*}
&\ll_{k,N} X\sum_{\text{max}(n_1,n_2)>M}r(n_1)r(n_2)|\lambda_0(n_1)\lambda_0(n_2)|\cdot|\mathcal G(1;n_1n_2)|\cdot\left|\frac{\lambda_g\left(\frac{n_1n_2}{(n_1,n_2)^2}\right)}{\left(\frac{n_1n_2}{(n_1,n_2)^2}\right)^{1/2}}\right|\\
&\ll_{k,N} X\sum_{n_1,n_2=1}^\infty r(n_1)r(n_2)|\lambda_0(n_1)\lambda_0(n_2)|\cdot|\mathcal G(1;n_1n_2)|\cdot\left|\frac{\lambda_g\left(\frac{n_1n_2}{(n_1,n_2)^2}\right)}{\left(\frac{n_1n_2}{(n_1,n_2)^2}\right)^{1/2}}\right|\left(\frac{n_1n_2}M\right)^\alpha
\end{align*}
for any $\alpha>0$, which is chosen optimally later.

Write $n=n_1n_2$ to express the double sum above as a single sum over $n$ and note that by the fact that $r(\cdot)$ is supported only on squarefree integers coprime to $N$ the only integers $n$ which contribute to the sum over $n$ satisfy $p^3 \nmid n$ and $(n,N)=1$. Hence, by multiplicativity of $r(\cdot)$ we can express the sum over $n$ as an Euler product and
the expression on the r.h.s of the previous display equals
\begin{align*}
&= XM^{-\alpha}\prod_{L^2\leq p\leq L^4}\left(1+2r(p)p^{\alpha-1/2}|\mathcal G(1;p)|\cdot\frac{|\lambda_0(p)\lambda_g(p)|}{\sqrt p}+r(p)^2p^{2\alpha}|\mathcal G(1;p^2)|\cdot|\lambda_0(p)|^2\right)\\
&\ll XM^{-\alpha}\exp\left(\sum_{L^2\leq p\leq L^4} \left(\frac{8Lp^{\alpha-1}}{\log p}+\frac{4L^2p^{2\alpha-1}}{(\log p)^2}\right) \left( 1+O\left( \frac1p\right)\right)\right),
\end{align*}
where the last estimate follows from Deligne's bound $|\lambda_g(p)|\leq 2$, the fact that $\mathcal G(1;p^k)=1+O(1/p)$, and the definition of $r(n)$.

Let us now choose $\alpha=1/(8\log L)$. Then by the prime number theorem and partial summation the above is
\[ \tst
\ll X\cdot\exp\left(-\frac{\log M}{8\log L}+\frac{8L}{\log L}+\frac{4L^2}{(\log L)^2}\right)\ll X
\]
by the choices $L=\frac18\sqrt{\log M\log\log M}$ and $M=X^{1/24}$. From the above arguments we deduce that
\begin{align}
\nonumber&\tst\sum_{d\in\mathcal D_{N,\eta}}
L\left(\frac12,g\otimes\chi_d\right)|R(d)|^2\Phi\left(\frac{|d|}X\right)\\
&\tst=\nonumber A_{g,\eta}(\Phi)\cdot X\cdot\prod_{L^2\leq p\leq L^4}\left(1+2r(p)\mathcal G(1;p)\cdot\frac{\lambda_0(p)\lambda_g(p)}{\sqrt p}+r(p)^2\mathcal G(1;p^2)\cdot \lambda_0(p)^2\right)+O_{k,N,\Phi}(X)\\ &\tst=A_{g,\eta}(\Phi)\cdot X\cdot\mathcal R\cdot\exp\left(2\sum_{L^2\leq p\leq L^4}\frac{r(p)\lambda_0(p)\lambda_g(p)}{\sqrt p}+O_{k,N,\Phi}\left(\frac L{(\log L)^3}\right)\right) + O_{k,N,\Phi}(X),
\end{align}
where the last step follows exactly as in \cite[\S 6]{gun2020large}.

We now apply this result with $g=g_f$, where $f\in\B_{k+\frac12,4\ell}^\new$ with $\ell|N$. For \eqref{estimate1} we choose $\Phi$ to be supported on the interval $[1,2]$ so that $\Phi(t)=1$ for $t\in[11/10,19/10]$ and satisfying $0\leq\Phi(t)\leq 1$. For \eqref{estimate2} we choose $\Phi$ to be supported in $[1/2, 5/2]$ with $\Phi(t)=1$ for $t\in[1,2]$ and again satisfying $0\leq\Phi(t)\leq 1$. These choices lead to
\begin{align*}
&\tst\sum_{d\in\mathcal D_{N,\eta}(X)}L\left(\frac12,g_0\otimes\chi_d\right)|R(d)|^2\\
&\tst\geq\frac{4X\mathcal R}{5N}\cdot L_{g_0,\eta}\left(\frac12\right)L\left(1,\text{Sym}^2g_0\right)\cdot\exp\left(2\sum_{L^2\leq p\leq L^4}\frac{r(p)\lambda_0(p)^2}{\sqrt p}+O\left(\frac L{(\log L)^3}\right)\right)+O_{k,N}(X)
\end{align*}
for $f=f_0$,
and
\begin{align*}
&\tst\sum_{d\in\mathcal D_{N,\eta}(X)}L\left(\frac12,g_f\otimes\chi_d\right)|R(d)|^2\\
&\tst\leq\frac{X\mathcal R}N\cdot L_{g_f,\eta}\left(\frac12\right)L\left(1,\text{Sym}^2g_f\right)\exp\left(2\sum_{L^2\leq p\leq L^4}\frac{r(p)\lambda_0(p)\lambda_{g_f}(p)}{\sqrt p}+O\left(\frac L{(\log L)^3}\right)\right)+O_{k,N}(X)
\end{align*}
for $f\neq f_0$.

To obtain \eqref{estimate1} and \eqref{estimate2} we use partial summation and the following estimate, which is a consequence of the Rankin-Selberg theory \cite[Theorem 3]{WuYe2007}:

\begin{lemma}
With the notation above we have
\[\sum_{p\leq x}\lambda_0(p)\lambda_{g_f}(p)\log p=x\cdot 1_{f=f_0}+o(x). \]
\end{lemma}

\section{Siegel cusp forms of degree 2}\label{s:main}
In this Section we  first review various properties of Siegel cusp forms of degree 2 and then go on to prove our main results stated in the introduction.
\subsection{Preliminaries}
Denote by $J$ the $4$ by $4$ matrix given by
$
J =
\left(\begin{smallmatrix}
0 & I_2\\
-I_2 & 0\\
\end{smallmatrix}\right)$ where $I_2$ is the identity matrix of size 2. Define the algebraic groups
   $\GSp_4$ and $\Sp_4$ over $\Z$ by
$$\GSp_4(R) = \{g \in \GL_4(R) \; | \; \T{g}Jg =
  \mu_2(g)J,\:\mu_2(g)\in R^{\times}\},$$
$$
\Sp_4(R) = \{g \in \GSp_4(R) \; | \; \mu_2(g)=1\},
$$
for any commutative ring $R$.  The Siegel upper-half space $\H_2$ of degree 2 is defined by
$$
\H_2 = \{ Z \in \mathrm{Mat}_{2\times 2}(\C)\;|\;Z =\T{Z},\ \Im(Z)
  \text{ is positive definite}\}.
$$

Let $k$ and $N$ be positive integers. Let $\Gamma^{(2)}_0(N) \subseteq \Sp_4(\Z)$ denote the Siegel congruence subgroup of  level $N$, i.e.,
\begin{equation}\label{defu1n}
 \Gamma^{(2)}_0(N) = \Sp_4(\Z) \cap \left(\begin{smallmatrix}\Z& \Z&\Z&\Z\\\Z& \Z&\Z&\Z\\N\Z& N\Z&\Z&\Z\\N\Z&N \Z&\Z&\Z\\\end{smallmatrix}\right).
\end{equation}

We define
$$
 g \langle Z\rangle = (AZ+B)(CZ+D)^{-1}\qquad\text{for }
 g=\left(\begin{smallmatrix} A&B\\ C&D \end{smallmatrix}\right) \in \Sp_4(\R),\;Z\in \H_2.
$$
We let $J(g,Z) = CZ + D$.
Let $S_k(\Gamma^{(2)}_0(N))$ denote the space of Siegel cusp forms of weight $k$ and level $N$; precisely, they consist of the holomorphic functions $F$ on
$\H_2$ which satisfy the relation
\begin{equation}\label{siegeldefiningrel}
F(\gamma \langle Z\rangle) = \det(J(\gamma,Z))^k F(Z)
\end{equation}
for $\gamma \in \Gamma_0^{(2)}(N)$, $Z \in \H_2$, and vanish at all the
cusps.
Any $F$ in $S_k(\Gamma^{(2)}_0(N))$ has a Fourier expansion \begin{equation}\label{siegelfourierexpansion}F(Z)
=\sum_{S \in \Lambda_2} a(F, S) e^{2 \pi i \Tr(SZ)}
\end{equation}
with $\Lambda_2$ defined in \eqref{e:Lambda2}. We have the relation \begin{equation}\label{fourierinvariance}a(F, T) = \det(A)^k\,a(F, \T{A}TA) \end{equation} for  $A \in \GL_2(\Z)$. In particular, the Fourier coefficient $a(F, T)$ depends only on the $\SL_2(\Z)$-equivalence class of $T$.
We say that a matrix $S \in \Lambda_2$ is fundamental if $\disc(S)= -4 \det(S)$ is a fundamental discriminant. Given a fundamental discriminant $d<0$ and $K= \Q(\sqrt{d})$, let $\Cl_K$ denote the ideal class group of $K$. It is well-known that the $\SL_2(\Z)$-equivalence classes of matrices in $\Lambda_2$ of discriminant $d$ can be identified with $\Cl_K$; so the expression $\sum_{S \in \Cl_K}a(F, S) \Lambda(S)$ makes sense for each $\Lambda \in \widehat{\Cl_K}$.

\subsection{Constructing half-integral weight forms}\label{s:constructhalf}
Each $F$ in $S_k(\Gamma^{(2)}_0(N))$ has a Fourier--Jacobi expansion
$F(Z) = \sum_{m > 0} \phi_m(\tau, z) e^{2 \pi i m \tau'}$ where we write $Z= \begin{pmatrix}\tau&z\\z&\tau' \end{pmatrix}$ and for each $m>0$,
\begin{equation}\label{jacobifourier}\phi_m(\tau, z) = \sum_{\substack{n,r \in \Z \\ 4nm> r^2}}a \left(F, \mat{n}{r/2}{r/2}{m}\right) e^{2\pi i (n \tau + r z)} \in J_{k,m}^{\text{cusp}}(N). \end{equation} Here $J_{k,m}^{\text{cusp}}(N)$ denotes the space of Jacobi cusp forms of weight $k$, level $N$ and index $m$.

Given a \emph{primitive} matrix $S = \mat{a}{b/2}{b/2}{c} \in \Lambda_2$ (i.e., $\gcd(a,b,c)=1$) we let $\P(S)$ denote the set of primes of the form $ax^2+bxy+cy^2$. The set $\P(S)$ is infinite; indeed by \cite[Theorem 1 (i)]{iwanprime}, \begin{equation}\label{e:iwaniec}\tst |\{p \in \P(S): p \le X\}| \gg_S  \frac{X}{\log X}.\end{equation}

For each prime $p$ dividing $N$, define the operator $U(p)$ acting on the space $S_k(\Gamma^{(2)}_0(N))$ via
\begin{equation}\label{upaction}a(U(p) F, S)=
 a(F, pS).\end{equation}

\begin{lemma}\label{l:yamana}Let $k>2$ be even and $N$ be squarefree. Let $F \in S_k(\Gamma^{(2)}_0(N))$ be an eigenfunction of the $U(p)$ operator for each prime $p|N$ \up{note that if $N=1$, this condition is trivially true}.

\begin{enumerate}
\item Then there exists $S_0 = \mat{a}{b/2}{b/2}{c} \in \Lambda_2$ such that $a(F,S_0) \neq 0$ and $d_0 = b^2-4ac$ is odd and squarefree \up{and hence, $S_0$ is fundamental}.

\item  Let $S_0$, $d_0$ be as above and let $p \nmid 2Nd_0$ be a prime such that $p \in \P(S_0)$.   Put $$h_p(z) =  \sum_{m=1}^\infty  a(m) e^{2 \pi i m z}, \quad \text{where } \quad a(m) = \sum_{\substack{0 \le \mu \le 2p-1 \\ \mu^2 \equiv -m \pmod{4p}}} a\left(F, \mat{\frac{m+\mu^2}{4p}}{\frac{\mu}{2}}{\frac{\mu}{2}}{p} \right). $$
Then $0 \neq h_p \in S^+_{k-\frac{1}{2}}(4pN)$. Furthermore, for each $m>0$ such that $\gcd(m, 4p)=1$,  \begin{equation}\label{e:connection}\tst a(m) = 2 a\left(F, \mat{\frac{m+\mu_0^2}{4p}}{\frac{\mu_0}{2}}{\frac{\mu_0}{2}}{p} \right)\end{equation} where $\mu_0$ is any integer satisfying $\mu_0^2 \equiv -m \pmod{4p}$; if no such $\mu_0$ exists then $a(m)=0$.
\end{enumerate}
\end{lemma}
\begin{proof}
The claim that there exists $S_0 = \mat{a}{b/2}{b/2}{c} \in \Lambda_2$ such that $a(F, S_0)\neq 0$ and $d_0 = b^2-4ac$ is odd and squarefree follows from Theorem 2.2 of \cite{sahaschmidt}.

Now let $p \in \P(S_0)$, $p \nmid 2Nd_0$. The fact that  $h_p \in S_{k-\frac{1}{2}}(4pN)$ follows from $\phi_{p} \in J_{k,p}^{\text{cusp}}(N) $ and Theorem 4.8 of~\cite{manickram}; by definition $h_p$ belongs to the Kohnen plus space.
Next, we prove  \eqref{e:connection}. Let $\gcd(m, 4p)=1$. Observe that the sum $\sum_{\substack{0 \le \mu \le 2p-1 \\ \mu^2 \equiv -m \pmod{4p}}}$ is non-empty if and only if $-m$ is a square modulo $4p$ in which case the sum has exactly two terms. Indeed, we get that  \begin{equation}\label{e:conn2}\tst a(m) = a\left(F, \mat{\frac{m+\mu_0^2}{4p}}{\frac{\mu_0}{2}}{\frac{\mu_0}{2}}{p}\right) + a\left(F, \mat{\frac{m+\mu_1^2}{4p}}{\frac{\mu_1}{2}}{\frac{\mu_1}{2}}{p}\right)\end{equation} where $0 \le \mu_0 \le 2p-1$ satisfies $\mu_0^2 \equiv -m \pmod{4p}$ and $\mu_1 = 2p-\mu_0$. Using \eqref{fourierinvariance}, \eqref{e:conn2}, and the identity $\mat{1}{-1}{0}{-1} \mat{\frac{m+\mu_0^2}{4p}}{\frac{\mu_0}{2}}{\frac{\mu_0}{2}}{p}\mat{1}{0}{-1}{-1} = \mat{\frac{m+\mu_1^2}{4p}}{\frac{\mu_1}{2}}{\frac{\mu_1}{2}}{p}$, we obtain \eqref{e:connection}.

 It remains to show $h_p \neq 0$ for which we will show that $a(-d_0)\neq 0$. Let $x_0$, $y_0$ be integers such that $cx_0^2 + bx_0y_0 + ay_0^2=p$.  Since $\gcd(x_0, y_0) = 1$, we may pick integers $x_1$, $y_1$ such that $A= \mat{y_1}{y_0}{x_1}{x_0} \in \SL_2(\Z).$ Then $S' = \T{A}SA$ is $\SL_2(\Z)$-equivalent to $S_0$ and has the form $S'=\mat{\frac{-d_0+\mu_0^2}{4p}}{\frac{\mu_0}{2}}{\frac{\mu_0}{2}}{p} \in \Lambda_2$. By \eqref{fourierinvariance},  $a(F,S') \neq 0$ since $a(F, S_0) \neq 0$. Hence by \eqref{e:connection},  $a(-d_0) = 2 a(F,S') \neq 0$.
\end{proof}

\subsection{The proofs of Theorems A and B}
We are now ready to prove Theorems A and B of the introduction in a slightly stronger form.
\begin{theorem}\label{thm:signssiegel}
Let $F \in S_k(\Gamma^{(2)}_0(N))$ with $k > 2$ even and $N$ odd and squarefree. Assume that $F$ is an eigenfunction of the $U(p)$ operator for each prime $p|N$, and that $F$ has real Fourier coefficients. Then there exists  a set $\P$ of primes satisfying $|\{p \in \P: p \le X\}| \gg_F  \frac{X}{\log X}$  such that given $\eps>0$ and $p \in \P$, there exist $M\ge 2$ \up{depending only on $F$ and $p$} and  $X_0\ge 1$ \up{depending on $F$, $p$ and $\eps$} so that for all $X \ge X_0$,
 there  are $r_X \ge X^{1-\eps}$ matrices $S_1$, $S_2, \ldots, S_{r_{X}} \in \Lambda_2$ satisfying the following:
\begin{enumerate}
\item $X \le |\disc(S_1)| < |\disc(S_2)| < \ldots < |\disc(S_{r_X})| \le MX$,
\item For each $1\le i \le r_X$, $S_i = \mat{*}{*}{\ast}{p}$, and $\disc(S_i)$ is a  squarefree integer coprime to $2N$,
\item For each $1\le i \le r_X-1$, $a(F, S_{i})a(F, S_{i+1})<0$.
\end{enumerate}
\end{theorem}
\begin{proof}Using Lemma \ref{l:yamana}, we fix a fundamental matrix $S_0$ such that $a(F,S_0) \neq 0$. Take $\P = \P(S_0)$; then the estimate $|\{p \in \P: p \le X\}| \gg  \frac{X}{\log X}$ follows from \eqref{e:iwaniec}. Given any $p \in \P$, let $f=h_{p}$ be as in Lemma \ref{l:yamana}, so that $0 \neq f \in S^+_{k-\frac{1}{2}}(4pN)$ and the coefficients of $f$ are all real since the coefficients of $F$ are.

By Theorem \ref{thm:signs}, there exists  $M\ge 2$ such that for any $\eps>0$, the sequence $\{ a(f,n)\mu^2(n) \}_{\substack{X \leq n \leq MX \\ (n,2N)=1}}$ has  $\ge C_{f,M,\varepsilon} X^{1-\varepsilon/2}$ sign changes for some constant $C_{f,M,\varepsilon}$. Pick $X_0 \ge 1$ so that for all $X \ge X_0$ we have $X^{\eps/2} \ge \frac{1}{C_{f,M,\varepsilon}}$. Then for all  $X \ge X_0$,
 there exists $r_X \ge X^{1-\eps}$, and an increasing sequence $(n_i)_{1 \le i \le r_X}$ of odd squarefree integers satisfying
$a(f,n_i) a(f,n_{i+1}) < 0$.
 For each $n_i$ as above,   \eqref{e:connection} tells us that  $a(f,n_i) = 2 a(F,S_i)$ for some $S_i = \mat{*}{*}{\ast}{p} \in \Lambda_2$  with  $|\disc(S_i)| = n_i$. This completes the proof.
\end{proof}
\begin{theorem}\label{thm:large}
Let $F \in S_k(\Gamma^{(2)}_0(N))$ with $k >2$ even and $N$ odd and squarefree. Assume that $F$ is an eigenfunction of the $U(p)$ operator for each prime $p|N$. Then there exists  a set $\P$ of primes satisfying $|\{p \in \P: p \le X\}| \gg_F  \frac{X}{\log X}$  such that given $\eps>0$ and $p \in \P$, one can find $X_0\ge 1$ (depending on $F$, $p$ and $\eps$) so that for all $X \ge X_0$,
 there  are $r_X \ge X^{1-\eps}$ matrices $S_1$, $S_2, \ldots, S_{r_{X}} \in \Lambda_2$ satisfying the following:
\begin{enumerate}
\item For each $1\le i \le r_X$, $S_i = \mat{*}{*}{\ast}{p}$, and $\disc(S_i)$ is a  squarefree integer coprime to $2N$,
\item $X \le |\disc(S_1)| < |\disc(S_2)| < \ldots < |\disc(S_{r_X})| \le 2X$,
\item For each $1\le i \le r_X$, $|a(F,S_i)| \ge |\disc(S_i)|^{\frac{k}2 - \frac34} \exp\left(\frac{1}{82} \sqrt{\frac{\log |\disc(S_i)|}{\log \log |\disc(S_i)|}} \right).$
\end{enumerate}
\end{theorem}
\begin{proof}The proof is essentially identical to Theorem \ref{thm:signssiegel}, except that we use Theorem \ref{thm:large2} (rather than Theorem \ref{thm:signs}) on $f = h_p$.
\end{proof}
\subsection{$L$-functions of Siegel cusp forms}\label{s:lfunctionssiegel}
For the rest of this section we assume that $k >2$.

 Given an irreducible cuspidal automorphic representation $\pi$ of $\GSp_4(\A)$, we let $L(s, \pi)$ denote the associated degree 4 $L$-function (known as the ``spin" $L$-function). Furthermore, we let $L(s, \std(\pi))$ denote the associated degree 5 $L$-function (the ``standard" $L$-function) and $L(s, \ad(\pi))$ denote the associated degree 10 $L$-function (the ``adjoint" $L$-function). Each of these $L$-functions is defined as an Euler product with the local $L$-factor at each prime (including the ramified  primes) constructed via the associated representation of the dual group $\GSp_4(\C)$ using the local Langlands correspondence (which is known for $\GSp_4$ by the work of Gan--Takeda \cite{gantakGSp4}). More precisely, for $n=4,5,10$, let $\rho_n$ denote the irreducible $n$-dimensional representation of $\GSp_4(\C)$ given as follows: $\rho_4$ is the inclusion of $\GSp_4(\C) \hookrightarrow \GL_4(\C)$, $\rho_5$ is the map defined in \cite[A.7]{NF}, and $\rho_{10}$ is the adjoint representation of $\GSp_4(\C)$ on the Lie algebra of $\Sp_4(\C)$. Then $L(s, \pi)$,  $L(s, \std(\pi))$, and $L(s, \ad(\pi))$ correspond to the representations $\rho_4$, $\rho_5$ and $\rho_{10}$ respectively.


 We say that an element $F$ of $S_k(\Gamma^{(2)}_0(N))$ \emph{gives rise to an irreducible representation} if its adelization (in the sense of \cite[\S3]{sahapet}) generates an irreducible cuspidal automorphic representation of $\GSp_4(\A)$. The automorphic representation  associated to any such $F$ is of trivial central character and hence may be viewed as an automorphic representation of $\PGSp_4(\A) \simeq \SO_5(\A)$. It can be checked  \cite[Prop. 3.11]{sahapet} that if $F$ gives rise to an irreducible representation then $F$ is an eigenform of the local Hecke algebras at all primes not dividing $N$. In addition, we say that such an $F$ is factorizable if its adelization corresponds to a factorizable vector in the  representation generated.

    We say that an irreducible cuspidal automorphic representation $\pi$ of $\GSp_4(\A)$ \emph{arises from $S_k(\Gamma^{(2)}_0(N))$ } if there exists $F \in S_k(\Gamma^{(2)}_0(N))$ whose adelization generates $V_\pi$ 
    (in which case, by definition, $F$ gives rise to the irreducible representation $\pi$, which therefore must be of trivial central character by the comments above).  We say that an irreducible cuspidal automorphic representation $\pi$ of $\GSp_4(\A)$ is CAP if it is nearly equivalent 
 to a constituent of a global induced representation of a proper parabolic subgroup of $\GSp_4(\A)$. If a CAP $\pi$ arises from $S_k(\Gamma^{(2)}_0(N))$, then by Corollary 4.5 of \cite{pitschram} it is associated to the Siegel parabolic (i.e., it is of \emph{Saito-Kurokawa type}). Such a $\pi$ is non-tempered at almost all primes (in particular, it violates the Ramanujan conjecture). On the other hand, if $\pi$ arises from $S_k(\Gamma^{(2)}_0(N))$ and is \emph{not} CAP, then $\pi$ satisfies the Ramanujan conjecture by a famous result of Weissauer \cite[Thm. 3.3]{weissram}.

  Thus, the space $S_k(\Gamma^{(2)}_0(N))$ has a natural decomposition into orthogonal subspaces $$S_k(\Gamma^{(2)}_0(N))=S_k(\Gamma^{(2)}_0(N))^{\rm  CAP}  \oplus S_k(\Gamma^{(2)}_0(N))^{ \rm T}$$ where $S_k(\Gamma^{(2)}_0(N))^{\rm  CAP}$ is spanned by forms $F$ which give rise to irreducible representations of Saito-Kurokawa type, and $S_k(\Gamma^{(2)}_0(N))^{ \rm T}$ is its orthogonal complement, spanned by forms $F$ which give rise to irreducible representations that are not of Saito-Kurokawa type. Furthermore, one can get a basis of each of the spaces $S_k(\Gamma^{(2)}_0(N))^{\rm  CAP}$ and $S_k(\Gamma^{(2)}_0(N))^{ \rm T}$ in terms of factorizable forms.  We refer the reader to  \cite[\S 3.1 and \S3.2]{DPSS15} for further comments related to the above discussion.

 If $\pi$ arises from $S_k(\Gamma^{(2)}_0(N))$  and is of Saito-Kurokawa type, then there exists a representation $\pi_0$ of $\GL_2(\A)$ and a Dirichlet character $\chi_0$ satisfying $\chi_0^2=1$ such that $L^N(s, \pi) = L^N(s, \pi_0) L^N(s+1/2, \chi_0)  L^N(s-1/2, \chi_0)$. Additionally, if $N$ is squarefree, then only $\chi_0=1$ is possible by a well-known result of Borel \cite{Borel} and so in this case we have $L^N(s, \pi) = L^N(s, \pi_0) \zeta^N(s+1/2)  \zeta^N(s-1/2)$.  There exists another typical situation where the $L$-function factors: we say that a $\pi$ arising from $S_k(\Gamma^{(2)}_0(N))$ is of \emph{Yoshida type} if there are representations $\pi_1$ and $\pi_2$ of $\GL_2(\A)$ such that  $L(s, \pi) = L(s, \pi_1) L(s, \pi_2)$; in this case one has (after possibly swapping $\pi_1$ and $\pi_2$) that $\pi_1$ arises from a classical holomorphic newform of weight 2 and $\pi_2$ arises from a classical holomorphic newform of weight $2k-2$ (see \cite[Sec. 4]{sahapet} for more details).  We let $S_k(\Gamma^{(2)}_0(N))^{ \rm Y}$ denote the subspace of $S_k(\Gamma^{(2)}_0(N))^{ \rm T}$ spanned by forms which give rise to an irreducible representation of Yoshida type, and let $S_k(\Gamma^{(2)}_0(N))^{ \rm G}$,  which represents the \emph{general type}, denote the orthogonal complement of $S_k(\Gamma^{(2)}_0(N))^{ \rm Y}$ in $S_k(\Gamma^{(2)}_0(N))^{ \rm T}$. So we get the following key orthogonal  decomposition into subspaces \begin{equation}\label{e:decompositionpackets}S_k(\Gamma^{(2)}_0(N))=S_k(\Gamma^{(2)}_0(N))^{\rm  CAP}  \oplus S_k(\Gamma^{(2)}_0(N))^{ \rm Y} \oplus S_k(\Gamma^{(2)}_0(N))^{ \rm G}.\end{equation} In the notation of \cite{schmidtpacket}, the three subspaces on the right side above correspond to the \emph{global Arthur packets} of type \textbf{(P)}, \textbf{(Y)} and \textbf{(G)} respectively. In the sequel we will be mostly concerned with the space $S_k(\Gamma^{(2)}_0(N))^{ \rm G}$ because the other two spaces are easier to handle. The following proposition collects together some relevant facts about the associated $L$-functions that follow from the work of Arthur \cite{Arthur2013}.

\begin{proposition}\label{p:lfunctionsfacts}Suppose that $F \in S_k(\Gamma^{(2)}_0(N))^{ \rm G}$ gives rise to an irreducible representation $\pi$.
\begin{enumerate}
\item\label{ass1} The representation $\pi$ has a strong functorial lifting to an irreducible \emph{cuspidal} automorphic representation $\Pi_4$ of $\GL_4(\A)$. In particular, if $\sigma$ is any irreducible automorphic representation of $\GL_n(\A)$, then we have an equality of degree $4n$ Rankin-Selberg $L$-functions $L(s, \pi \times \sigma)= L(s, \Pi_4 \times \sigma)$ and therefore $L(s, \pi \times \sigma)$ satisfies the usual properties\footnote{By this we mean that this $L$-function has meromorphic continuation to the entire complex plane, satisfies the standard functional equation taking $s \mapsto 1-s$, has an Euler product, and is bounded in vertical strips (in particular, the $L$-function is in the extended Selberg class).}  of an $L$-function. If $n \le 3$,  then $L(s, \pi \times \sigma)$ is entire.

\item   The representation $\pi$ has a strong functorial lifting to an irreducible automorphic representation $\Pi_5$ of $\GL_5(\A)$. In particular, if $\sigma$ is any irreducible automorphic representation of $\GL_n(\A)$, then we have an equality of degree $5n$ Rankin-Selberg $L$-functions $L(s, \std(\pi) \times \sigma)=  L(s, \Pi_5 \times \sigma)$ and therefore $L(s, \std(\pi) \times \sigma)$ satisfies the usual properties  of an $L$-function. If $n \le 2$, and $\sigma$ has the property that the set of finite primes where it is ramified is either empty or contains at least one prime not dividing $N$, then $L(s, \std(\pi) \times \sigma)$ is entire.

\item   The degree 10 $L$-function      $L(s, \ad(\pi))$ satisfies the usual properties  of an $L$-function, is entire, and has no zeroes on the line $\Re(s)=1$.

\end{enumerate}
\end{proposition}
 \begin{proof} Since $F$ corresponds to the ``general" Arthur parameter,  Arthur's work \cite{Arthur2013} (see also Section 1.1 of \cite{schmidtpacket}) shows that $\pi$ has a strong lifting to a cuspidal automorphic representation $\Pi_4$ of $\GL_4$; clearly $\Pi_4$ has trivial central character since $\pi$ does. It is known from \cite[Thm. 1.5.3]{Arthur2013} that $\Pi_4$  is self-dual and symplectic\footnote{Recall that a self-dual representation $\Pi$ of $\GL_n(\A)$ is said to be symplectic if $L(s, \Sym^2 \Pi)$ is holomorphic at $s=1$.}. The fact that the lifting is ``strong", i.e., corresponds to a local lift at \emph{all} places, implicitly uses that the local parameters of Arthur coincide with the local Langlands parameters of Gan--Takeda \cite{gantakGSp4}; this consistency of local parameters follows from \cite{GC15}.  The required properties of $L(s, \pi \times \sigma)$ now follow from Rankin--Selberg theory of $\GL_n$. 

 For the second assertion, let $\wedge^2$ denotes the exterior square, and note that $\wedge^2\rho_4=\mathbf{1}+\rho_5$, hence
\begin{equation}\label{exteriorsquaredeg5relationeq}
 L(s,\Pi_4,\wedge^2)=L(s, \std(\pi))\zeta(s).
\end{equation}
 Recall that $L(s, \Pi_4 \times  \Pi_4) = L(s,\Pi_4,\Sym^2) L(s,\Pi_4,\wedge^2)$. Since $\Pi_4$ is symplectic, $L(s,\Pi_4,\wedge^2)$ has a simple pole at $s=1$ and $L(s,\Pi_4,\Sym^2)$ is an entire function by \cite[Thm.\ 7.5]{BG}. It follows from \eqref{exteriorsquaredeg5relationeq} that $L(s, \std(\pi))$ is holomorphic and non-zero at $s=1$. Together with \cite[Theorem 2]{Gr}, we obtain that $L(s, \std(\pi))$ has no poles on ${\rm Re}(s)=1$. On the other hand, by \cite[Thm. A]{kimsar} and \cite{Hen09} $L(s,\Pi_4,\wedge^2)$ is the $L$-function of an automorphic representation of $\GL_6(\A)$ of the form
${\rm Ind}(\tau_1\otimes\ldots\otimes\tau_m)
$
where $\tau_1,\ldots,\tau_m$ are unitary, cuspidal, automorphic representations of $\GL_{n_i}(\A)$, $n_1+\ldots+n_m=6$. Since $L(s,\Pi_4,\wedge^2)$ has a simple pole at $s=1$, it follows that exactly one of the $\tau_i$, say $\tau_m$, is the trivial representation of $\GL_1(\A)$. Cancelling out one zeta factor, we see that
\begin{equation}\label{liftingGL5theoremeq2}
 L(s, \std(\pi))=L(s,\tau_1)\cdots L(s,\tau_{m-1}).
\end{equation}
So we take $\Pi_5={\rm Ind}(\tau_1\otimes\ldots\otimes\tau_{m-1})
$; the above discussion shows that $\Pi_5$ is a (strong) lifting of $\pi$ from $\GSp_4(\A)$ to $\GL_5(\A)$ with respect to the  map $\rho_5$ of dual groups $\GSp_4(\C) \rightarrow \SO_5(\C) \subset \GL_5(\C)$. Observe that each $\tau_i$ is unitary, cuspidal, and unramified outside primes dividing $N$.  By Rankin--Selberg theory, $L(s, \Pi_5 \times \sigma)=\prod_{i=1}^{m-1}L(s, \tau_i \times \sigma) $ satisfies the usual properties of an $L$-function. To complete the proof, it suffices to show that if $\sigma$ is an  irreducible cuspidal representation of $\GL_1(\A)$ or $\GL_2(\A)$ such that the set of ramification primes for $\sigma$ is either empty or contains at least one prime outside $N$, then  $\tau_i \not\simeq \hat{\sigma}$ for each $i$. First, consider the case that $\sigma$ is a character, in which case we are reduced to $n_i=1$.  In this case, since $\tau_i$ is unramified outside $N$, the assumption on $\sigma$ means that the situation  $\tau_i \simeq \hat{\sigma}$ will force $\tau_i$ to be unramified everywhere, in which case  the right hand side of (\ref{liftingGL5theoremeq2}) would have a pole on ${\rm Re}(s)=1$. This contradicts the observation from above that $L(s, \std(\pi))$ has no poles on ${\rm Re}(s)=1$. Next consider the case that $\sigma$ is a cuspidal representation of $\GL_2(\A)$. An identical argument to the above reduces us to the case where  $\tau_i \simeq \hat{\sigma}$ is unramified at all finite primes. The easy relation $\wedge^2\rho_5 = {\rm Sym}^2\rho_4=\rho_{10}$  implies that
\begin{equation}\label{liftingGL5theoremeq3}
 L(s,\Pi_5,\wedge^2)= L(s, \ad(\pi)) = L(s,\Pi_4,\Sym^2),
\end{equation}
which we know is an entire function. On the other hand, if some $\tau_i$ in \eqref{liftingGL5theoremeq2} has $n_i=2$ and is unramified at all finite primes, then $L(s,\Pi_5,\wedge^2)$ will contain a factor of $L(s, \omega_{\tau_i})$ and so will have a pole on $\Re(s)=1$. This contradiction completes the proof that each $L(s,\tau_i \times \sigma)$ and hence $L(s,\Pi_5\times \sigma)$  is entire.

Finally, the assertion concerning the degree 10 $L$-function $L(s, \ad(\pi))$ follows from the identity  \eqref{liftingGL5theoremeq3} and the fact that $L(s,\Pi_4,\Sym^2)$ represents a holomorphic $L$-function. Since symmetric square $L$-functions are accessible via the Langlands-Shahidi method, the non-vanishing on ${\rm Re}(s)=1$ follows (see Sect.\ 5 of \cite{Sh1981} and \cite[Thm 1.1]{shahidi97}).
 \end{proof}
For our future applications, we will also need that  $L(s, \ad(\pi) \times \sigma)$ has the properties of an $L$-function and is entire for certain special automorphic representations $\sigma$ of $\GL_1(\A)$ or $\GL_2(\A)$. The next two lemmas achieve this for $\sigma$ a quadratic character or $\sigma$ of the form $\AI(\Lambda^2)$ where $\Lambda$ is a character of $\A_K^\times$ where $K$ is a quadratic field and $\AI$ denotes automorphic induction.
 \begin{lemma}\label{l:basechange}Let $K/\Q$ be a quadratic field. Let $F$ and $\pi$ be as in Proposition \ref{p:lfunctionsfacts} and assume that $N$ is squarefree. Then the base change  $\pi_K$ of $\pi$ to $\GSp_4(\A_K)$ is cuspidal. Furthermore, the base change $\Pi_{4,K}$ of $\Pi_4$ to  $\GL_4(\A_K)$  is cuspidal and $\Pi_{4,K}$ is the lifting of $\pi_K$ from $\GSp_4(\A_K)$ to $\GL_4(\A_K)$.
 \end{lemma}
 \begin{proof}Let $\pi_K$ and $\Pi_{4,K}$ be as in the statement of the lemma. Since $\Pi_4$ is the lifting of $\pi$, it follows from the definition of base change that $\Pi_{4,K}$ is the lifting of $\pi_K$.  Let $\sigma$ be an arbitrary cuspidal automorphic representation of $\GL_1(\A_K)$ or $\GL_2(\A_K)$. By the definition of lifting, we have \[L(s, \pi_K \times \sigma) = L(s, \Pi_{4,K} \times \sigma)\] and so to prove that $\pi_K$ and $\Pi_{4,K}$ are cuspidal it suffices to show that $L(s, \Pi_{4,K} \times \sigma)$ has no poles. By the adjointness formula \cite[Prop 3.1]{PrasadRamakrishnan1999} we have $L(s, \Pi_{4,K} \times \sigma) = L(s, \Pi_4 \times \AI(\sigma))$. Note that $\AI(\sigma)$ is an automorphic representation of $\GL_2(\A)$ if $\sigma$ is a character of $\A_K^\times$ and $\AI(\sigma)$ is an automorphic representation of $\GL_4(\A)$ if $\sigma$ is a cuspidal representation of $\GL_2(\A_K)$. By Proposition \ref{p:lfunctionsfacts}, $L(s, \Pi_4 \times \AI(\sigma))$  is entire when $\sigma$ is a character. This reduces us to the case that $\sigma$ is a cuspidal representation of $\GL_2(\A_K)$; in this case $L(s, \Pi_{4,K} \times \sigma) =  L(s, \Pi_4 \times \AI(\sigma))$ has a pole if and only if  $\Pi_4 \simeq \AI(\sigma)$ (recall that $\Pi_4$ is self-dual). However, by looking at a prime $p$ which ramifies in $K$, we see that this is impossible: at any such prime, the local Langlands parameter of $\Pi_{4,p}$ is the local lifting of a representation of $\GSp_4(\Q_p)$ that has a vector fixed by the local Siegel congruence subgroup of level $p$ (since $N$ is squarefree), but an inspection of Table 1 of \cite{JL-R12} tells us that the local parameter of $\AI(\sigma_p)$ can never equal one of those. Hence we have completed the proof that $L(s, \Pi_{4,K} \times \sigma)$ has no poles.
  \end{proof}
 \begin{remark}The above proof crucially relies on the fact that $F$ has squarefree level $N$. If $N$ is allowed to be divisible by squares of primes, there indeed exists $F$ whose adelization generates a representation $\pi$ whose base change $\pi_K$ (for certain $K$) is non-cuspidal. Such $F$ are constructed in \cite{JL-R12} for $K$ real quadratic and \cite{BDPS15} for $K$ imaginary quadratic.
 \end{remark}

 \begin{lemma}\label{l:adjprops}Let $K/\mathbb Q$ be a quadratic field. Let $F$ and $\pi$ be as in Proposition \ref{p:lfunctionsfacts} and assume that $N$ is squarefree. Let $\chi_K$ be the quadratic Dirichlet character associated to the extension $K/\Q$. Let $\Lambda$ be any idele class character of $K^\times \bs A_K^\times$. Then the degree 10 $L$-function $L(s, \ad(\pi)\times \chi_K)$ and the degree 20 $L$-function $L(s, \ad(\pi)\times \AI(\Lambda^2))$ both satisfy the usual properties of an $L$-function, and are both entire.
 \end{lemma}
 \begin{proof}It can be verified (as a formal identity that holds at each place) that \[L(s, \Sym^2(\pi_{4,K} \times \Lambda)) = L(s, \ad(\pi) \times \AI(\Lambda^2)).\] By Lemma \ref{l:basechange}, $\pi_{4,K} \times \Lambda$ is a cuspidal representation of $\GL_4(\A_K)$.  Furthermore, if $\pi_{4,K} \times \Lambda$ is self-dual, then it must be symplectic since $\Pi_{4,K}\times \Lambda$ is the lift of $\pi_K\times \Lambda$ from $\GSp_4(\A_K)$ to $\GL_4(\A_K)$. So by \cite[Thm.\ 7.5]{BG}, $ L(s, \ad(\pi) \times \AI(\Lambda^2))$ has the properties of an $L$-function and is entire.

 To show the same fact for $L(s, \ad(\pi)\times \chi_K) = L(s, \Sym^2(\Pi_4)\times \chi_K)$ we appeal to \cite[Thm. p. 104]{Takeda}, which shows that the only possible poles of $L(s, \Sym^2(\Pi_4)\times \chi_K)$ are at $s=0, 1$ (these two possible poles are related by the functional equation). However, by what we have already proved, $L(s,\ad(\pi)\times \AI(1)) =L( s, \ad(\pi))L(s, \ad(\pi)\times \chi_K)$ is entire. So a pole at $s=1$ for $L(s, \ad(\pi)\times \chi_K)$ will imply that $L(1, \ad(\pi))=0$. This contradicts the last part of Proposition \ref{p:lfunctionsfacts}.
 \end{proof}

\begin{remark}\label{r:holofunctions}Let $F \in S_k(\Gamma^{(2)}_0(N))^{ \rm G}$ give rise to an irreducible representation $\pi$ and assume that $N$ is squarefree. Let $d$ be a fundamental discriminant that is divisible by at least one prime not dividing $N$, put $K=\Q(\sqrt{d})$, and let $\chi_d$ be the quadratic Dirichlet character associated to $K/\Q$ and let $\Lambda$ be a character of the ideal class group $\Cl_K$. Then combining Proposition \ref{p:lfunctionsfacts} and Lemma \ref{l:adjprops}, we see that the $L$-functions $L(s, \pi \times \AI(\Lambda))$, $L(s, \ad(\pi) \times \AI(\Lambda^2))$, $L(s, \ad(\pi) \times \chi_d)$, $L(s, \std(\pi) \times \AI(\Lambda^2))$, $L(s, \std(\pi) \times \chi_d)$ are all holomorphic everywhere in the complex plane.
\end{remark}

\subsection{A consequence of the refined Gan--Gross--Prasad identity}

Let $K = \Q(\sqrt{d})$ where $d<0$ is a fundamental discriminant. Recall from the introduction that for any character $\Lambda$ of the finite group $\Cl_K$, we define $B(F, \Lambda) = \sum_{S \in \Cl_K}a(F, S) \Lambda(S)$. The refined Gan--Gross--Prasad (GGP) conjecture in the context of Bessel periods of $\PGSp_4 \simeq \SO_5$ (and more generally, for Bessel periods of orthogonal groups)  was formulated in \cite[(1.1)]{yifengliu}, and made more explicit in the context of Siegel cusp forms in \cite{DPSS15}. This conjecture implies an identity expressing the square of $|B(F, \Lambda)|$ as a ratio of $L$-values, up to some local integrals. For the purpose of this paper, we only need a relatively weak consequence of this identity, which we formulate explicitly as \emph{Hypothesis G} below. In this subsection we do not assume that $N$ is squarefree.
\medskip

\emph{Suppose that $F \in  S_k(\Gamma^{(2)}_0(N))^{ \rm T}$ gives rise to an irreducible representation $\pi$. Then $F$ is said to satisfy \textbf{Hypothesis G} if there exists a constant $C_F$ such that for each imaginary quadratic field $K=\Q(\sqrt{d})$ \up{with $d<0$ a fundamental discriminant} and each character $\Lambda$ of $\Cl_K$ we have \begin{equation}\label{e:ggpineq}\tst |B(F, \Lambda)|^2  \le C_F |d|^{k-1} L(\tfrac12, \pi \times \AI(\Lambda)).\end{equation}}

\begin{proposition}\label{p:ggpimpliesG} Let $\pi$ be an irreducible representation that arises from $S_k(\Gamma^{(2)}_0(N))^{ \rm T}$. Suppose that for each factorizable $F \in  S_k(\Gamma^{(2)}_0(N))^{ \rm T}$ that gives rise to $\pi$, and each ideal class character $\Lambda$ of an imaginary quadratic field $K=\Q(\sqrt{d})$ \up{with $d$ a fundamental discriminant},  the refined Gan--Gross--Prasad identity \up{in the form written down in Conjecture 1.3 of \cite{DPSS15}} holds for $(\phi, \Lambda)$, where $\phi$ is the adelization of $F$. Then any $F \in  S_k(\Gamma^{(2)}_0(N))^{ \rm T}$ that gives rise to $\pi$ satisfies Hypothesis G.
\end{proposition}
\begin{proof} Note that the subspace of $S_k(\Gamma^{(2)}_0(N))^{ \rm T}$ generated by forms that give rise to a fixed irreducible representation $\pi$, has a basis consisting of factorizable forms. So, for the purpose of verifying Hypothesis G, it suffices to prove \eqref{e:ggpineq} for factorizable $F$ whose adelization $\phi = \otimes \phi_v$ generates $\pi$. Assuming the truth of Conjecture 1.2  of \cite{DPSS15} for $(\phi, \Lambda)$, and combining it with the explicit calculations performed in \cite[\S2.2, \S3.3-3.5]{DPSS15}, we see that \[\tst \left|B(F, \Lambda)\right|^2  = A_F |d|^{k-1} L(\tfrac12, \pi \times \AI(\Lambda))  \prod_{p|N_F}J_p(\phi_p, \Lambda_p) \] where $A_F$ depends on $F$, $N_F|N$ is the smallest integer such that $F \in S_k(\Gamma^{(2)}_0(N_F))^{ \rm T}$, and the normalized local factors $J_p(\phi_p, \Lambda_p)$ (which depend on $p$, $\phi_p$, $\Lambda_p$ and  $d$) are defined in \cite[(30)]{DPSS15}. To complete the proof, it suffices to show that  $J_p(\phi_p, \Lambda_p) \ll_{p,\phi_p} 1$ (i.e., is bounded by some constant that depends on $p$, $\phi_p$ but not on $d$ or $\Lambda_p$). It suffices to show this for the unnormalized local factors $J_{0,p}(\phi_p, \Lambda_p)$ defined in \cite[(29)]{DPSS15} since the normalized local $L$-factors $J_p(\phi_p, \Lambda_p)$ only differ from these by certain absolutely bounded $L$-factors appearing in \cite[(30)]{DPSS15}.

To show the above, we move to a purely local setup. Let $p$ be a prime dividing $N_F$.  Let $F =\Q_p$.  We fix a set $M$ of coset representatives of $F^\times / (F^\times)^2$ such that all elements of $M$ are taken from $\Z_p$, and each $r \in M$ generates the discriminant ideal of $F(\sqrt{r})/F$. We let $r_p$ equal the unique representative in $M$ that corresponds to $d$. The assumptions imply that \[r_p = d u_p^2, \quad \text{ for some }u \in (\Z_p^\times)^2.\]
Let $K$ equal  $F\times F$ (the ``split case") if $r_p \in F^2$ and $K=F(\sqrt{r_p})= F(\sqrt{d})$  (the ``field case") if $r_p \notin F^2$. Fix the matrix $S=S_d$ as in \cite[(75)]{DPSS15}, so that $S_d$ has discriminant $d$. Let $T(F)=T_S(F) \simeq K^\times$ be the associated subgroup of $\GSp_4(F)$, let $N(F)\subset \GSp_4(F)$ denote the unipotent for the Siegel parabolic and $\theta_S$ the character on $N(F)$ given by $\theta_S(\mat{I_2}{X}{0}{I_2}) = \psi(\mathrm{tr}(SX))$ where $\psi$ is a fixed unramified additive character.  Then we need to show that the integral $$J_{0,p}(\phi_p, \Lambda_p) =\int_{F^\times \bs T(F)}\int^{\St}_{N(F)}\langle \pi_p(t_pn_p) \phi_p , \phi_p \rangle \Lambda_p^{-1}(t_p) \theta_S^{-1}(n_p)\,dn_p\,dt_p$$ is bounded by some quantity that does not depend on $d$ or $\Lambda_p$. Note above that the superscript $\St$ denotes stable integral as in \cite{yifengliu}; this means that the integral can be replaced by any sufficiently large compact subgroup (as we will do below). Put $S'=\mat{-\frac{r_p}4}{0}{0}{1}$,
$A' = \mat{u_p}{0}{0}{1}$ if $d \equiv 0 \pmod{4}$. In the case that $d \equiv 1 \pmod{4}$ we put $S' = \mat{\frac{1-r_p}4}{\frac12}{\frac12}{1}$, $A' = \mat{u_p}{0}{\frac{1-u_p}{2}}{1}$. In either case,  $A' \in \GL_2(\Z_p)$ and  $S'=\T{A}SA$. To show that $J_p$ is bounded independently of $d$, we use a simple change of variables
  ($n_p \mapsto An_p\T{A}$, $t_p \mapsto At_p'A^{-1}$) to see that the integral $J_{0,p}(\phi_p, \Lambda_p)$ remains unchanged when the matrix $S$ is replaced by $S'$. This shows that $J_{0,p}(\phi_p, \Lambda_p)$ does not depend on the actual value of $d$ but only on the class of $d$ in $F^\times / (F^\times)^2$ together with $d$ modulo 4,  of which there are only finitely many possibilities. To show that the resulting integral is absolutely bounded independent of $\Lambda$, we replace the integral $\int^{\St}_{N(F)}$ by $\int_{p^{-m_p}N(\Z_p)}$ where $m_p$ depends only on the level $N$ (as follows from the argument of \cite[Prop. 2.14]{DPSS15}). The resulting integral is absolutely convergent \cite[Thm. 2.1 (i)]{yifengliu} and hence bounded independently of $\Lambda$.
\end{proof}
\begin{remark}Furusawa and Morimoto  \cite{furmori} have proved the refined GGP identity in the form required in Prop. \ref{p:ggpimpliesG} for $\Lambda=1_K$, and they have announced a proof of this identity for general $\Lambda$.
\end{remark}

\begin{remark}Under certain assumptions (namely $N$ odd and squarefree, $F$ a newform, and $\big( \frac{d}{p} \big) = -1$ for all primes $p$ dividing $N$) the relevant local integrals were explicitly computed in \cite{DPSS15}, and so under these assumptions Proposition \ref{p:ggpimpliesG} follows from Theorem 1.13 of \cite{DPSS15}.
\end{remark}

\begin{corollary}\label{c:yoshidahyp}Suppose that $F \in  S_k(\Gamma^{(2)}_0(N))^{ \rm Y}$ gives rise to an irreducible representation. Then $F$ satisfies Hypothesis G.
\end{corollary}
\begin{proof} This follows from Proposition \ref{p:ggpimpliesG} since the refined Gan--Gross--Prasad conjecture is known for Yoshida lifts by \cite[Thm 4.3]{yifengliu}.
\end{proof}

\begin{proof}[The proof of Corollary \ref{c:lvalueslarge}]
Let $\pi$ be a cuspidal automorphic representation of $\GSp_4(\A)$ that arises from a form $F \in S_k(\Gamma_0^{(2)}(N))^T$ with $k>2$ even and $N$ odd and squarefree. Assume that the refined Gan--Gross--Prasad identity holds. Then for any fundamental matrix $S \in \Lambda_2$ of discriminant $d$, we put $K=\Q(\sqrt{d})$ and use Cauchy-Schwarz and Proposition \ref{p:ggpimpliesG} to conclude \begin{align*}\tst|a(F,S)|^2 &= \frac{1}{|\Cl_K|^2}\left|\sum_{\Lambda \in \widehat{\Cl_K}} B(F, \Lambda) \Lambda^{-1}(S)\right|^2  \ \le   \frac{1}{|\Cl_K|}\sum_{\Lambda \in \widehat{\Cl_K}} |B(F, \Lambda)|^2  \\ \tst &\ll_{\pi} |d|^{k-1}\frac{1}{|\Cl_K|}\sum_{\Lambda \in \widehat{\Cl_K}}  L(\tfrac12, \pi \times \AI(\Lambda)).
\end{align*} The proof of the Corollary now follows from Theorem B.
\end{proof}

\begin{proof}[The proof of Corollary \ref{c:yoshidavalues}]For $f$, $g$ as in Corollary \ref{c:yoshidavalues}, and $N=\mathrm{lcm}(N_1, N_2)$ there exists $F \in S_k(\Gamma_0^{(2)}(N))^{\rm{Y}}$  which gives rise to an irreducible representation $\pi$ such that (see \cite[Prop. 3.1]{sahaschmidt}) \[L(s, \pi \times \AI(\Lambda)) = L(s, f \times \AI(\Lambda))L(s, g \times \AI(\Lambda)).\] Now Corollary \ref{c:yoshidavalues} follows from Corollary \ref{c:lvalueslarge} and the triangle inequality.
\end{proof}

\begin{proposition}\label{p:saito-kurokawa}Suppose that $F \in  S_k(\Gamma^{(2)}_0(N))^{ \rm CAP}$ gives rise to an irreducible representation $\pi$ such that $L^N(s, \pi)= L^N(s, \pi_0)\zeta^N(s+\frac12)\zeta^N(s-\frac12)$ where $\pi_0$ is a representation of $\GL_2(\A)$. Then the Fourier coefficients of $F$ at fundamental matrices depend only on the associated discriminant, i.e., if $S_1$, $S_2$ are two fundamental matrices with the same discriminant $d$, then $a(F, S_1) = a(F, S_2)$. Furthermore, there exists a constant $C_F$ such that for each fundamental discriminant $d$  and any matrix $S$ of discriminant $d$, we have \[|a(F, S)|  \le C_F |d|^{\frac{k}{2}-\frac34} L(\tfrac12, \pi_0 \times \chi_d)^{1/2}.\]
\end{proposition}
\begin{proof}An analogue of the Gan--Gross--Prasad conjecture for representations of Saito-Kurokawa type was formulated and proved by Qiu in \cite{qiu}. In particular, we have $B(F, \Lambda)=0$ unless $\Lambda$ is the trivial character, which proves the assertion that the Fourier coefficients of $F$ at fundamental matrices depend only on the discriminant. For $\Lambda$ trivial, the desired inequality follows in a similar fashion as in the proof of Proposition \ref{p:ggpimpliesG}.
\end{proof}

\subsection{The proof of Theorem C}
In this subsection we restate and prove Theorem C, assuming the main result of Section \ref{s:fractional}.

\begin{theorem}\label{thm:upperbd}Let $k >2$ be even  and $N\ge 1$ be  squarefree. \begin{enumerate}
\item Assume that the GRH holds for each of the $L$-functions $L(s, \pi \times \AI(\Lambda))$, $L(s, \ad(\pi) \times \AI(\Lambda^2))$, $L(s, \ad(\pi) \times \chi_d)$, $L(s, \std(\pi) \times \AI(\Lambda^2))$, $L(s, \std(\pi) \times \chi_d)$, $L(s, \sigma \times \chi_d)$ and $L(s,\chi_d)$, where $\pi$ is any automorphic representation of $\GSp_4(\A)$ that arises from $S_k(\Gamma_0^{(2)}(N))^{\rm G}$, $\sigma$ is any cuspidal automorphic representation of $\GL_2(\A)$, $K=\Q(\sqrt{d})$ is an imaginary quadratic field with associated Dirichlet character $\chi_d$, and $\Lambda$ is an ideal class character of $K$.

\item  Assume that each $F' \in  S_k(\Gamma^{(2)}_0(N))^{ \rm G}$ that gives rise to an irreducible representation  satisfies Hypothesis G \up{by Proposition \ref{p:ggpimpliesG} this is implied by the refined GGP conjecture}.
\end{enumerate}

Then, given any $F \in S_k(\Gamma^{(2)}_0(N))$ we have for fundamental matrices $S$, \begin{equation}\label{e:toprove}|a(F,S)| \ll_{F, \eps} \frac{|\disc(S)|^{\frac{k}2 - \frac{1}{2}}}{ (\log |\disc(S)|)^{\frac18 - \eps}}.\end{equation}
\end{theorem}
\begin{remark}\label{r:automorphy}Concerning the automorphy of the various representations appearing above, we refer the reader to the statements and proofs of Proposition~\ref{p:lfunctionsfacts} and Lemma~\ref{l:adjprops}.
\end{remark}
\begin{proof}Since $S_k(\Gamma^{(2)}_0(N))$ has a basis consisting of forms that give rise to irreducible representations, it suffices to prove \eqref{e:toprove} for such forms $F$. Furthermore, we may assume that $d=\disc(S)$ is divisible by at least one prime not dividing $N$, since there are at most finitely many fundamental $S$ without this property. Let $F$ give rise to $\pi$. We consider three cases, corresponding to  \eqref{e:decompositionpackets}.

\emph{Case 1}: \emph{$\pi$ is of Saito-Kurokawa type.} In this case, using Proposition \ref{p:saito-kurokawa}, the desired inequality follows from any subconvexity bound for $L(\tfrac12, \pi_0 \times \chi_d)$ which is known. (In particular, for this case, we do not need to appeal to GRH for $L(s, \pi_0 \times \chi_d)$).

\emph{Case 2}: \emph{$\pi$ is of Yoshida type, i.e., $F \in S_k(\Gamma^{(2)}_0(N))^{\rm Y}$.}
In this case, there are inequivalent representations $\sigma_1$ and $\sigma_2$ of $\GL_2(\A)$ such that $L(\tfrac12, \pi \times \AI(\Lambda)) = L(\tfrac12, \sigma_1 \times \AI(\Lambda)) \ L(\tfrac12, \sigma_2 \times \AI(\Lambda)).$ Using the identity  $a(F,S) = \frac{1}{|\Cl_K|}\sum_{\Lambda \in \widehat{\Cl_K}} B(F, \Lambda) \Lambda^{-1}(S)$ and Corollary \ref{c:yoshidahyp},
we conclude that \[\tst a(F,S) \ll_F |d|^{\frac{k}{2} - \frac12} \frac{1}{|\Cl_K|}\sum_{\Lambda \in \widehat{\Cl_K}} \sqrt{L(\tfrac12, \sigma_1 \times \AI(\Lambda)) \ L(\tfrac12, \sigma_2 \times \AI(\Lambda))}.\] Now the desired conclusion follows from Theorem 3 of \cite{Blomer-Brumley}.

\emph{Case 3}: \emph{$\pi$ is of general type, i.e., $F \in S_k(\Gamma^{(2)}_0(N))^{\rm G}$.}

In this case, using $a(F,S) = \frac{1}{|\Cl_K|}\sum_{\Lambda \in \widehat{\Cl_K}} B(F, \Lambda) \Lambda^{-1}(S)$ and Hypothesis G, we get \[\tst a(F,S) \ll_F |d|^{\frac{k}{2} - \frac12} \frac{1}{|\Cl_K|}\sum_{\Lambda \in \widehat{\Cl_K}} \sqrt{L(\tfrac12, \pi \times \AI(\Lambda))}.\] Now the desired conclusion follows from Theorem \ref{thm:momentbd} which will be proved in the next section.
\end{proof}

\section{Fractional moments of central $L$-values} \label{s:fractional}

\subsection{Statement of result}
Throughout this section, we let $\pi$ denote an irreducible cuspidal automorphic representation of $\GSp_4(\A)$ that arises from $S_k(\Gamma_0(N))^{\rm G}$ (see Section \ref{s:lfunctionssiegel}). Assume that $N\ge 1$ is squarefree  and $k \ge 2$. Let $K=\mathbb Q(\sqrt{d})$ be a quadratic field such that $d<0$ is a fundamental discriminant divisible by at least one prime not dividing $N$ and let $\tmop{Cl}_K$ denote the ideal class group of $K$. Also, given $\Lambda \in \widehat{ \tmop{Cl}_K}$ we let $\AI(\Lambda)$ denote the automorphic representation of $\GL_2(\A)$ given by automorphic induction; it is generated by the theta series $\Theta_{\Lambda}(z) = \sum_{ 0 \neq \mathfrak a \subset \mathcal O_K} \Lambda(\mathfrak a) e(N(\mathfrak a) z)
.$
Our assumptions imply that the $L$-functions $L(s, \pi \times \AI(\Lambda))$, $L(s, \ad(\pi) \times \AI(\Lambda^2))$, $L(s, \ad(\pi) \times \chi_d)$, $L(s, \std(\pi) \times \AI(\Lambda^2))$, $L(s, \std(\pi) \times \chi_d)$ are all entire (see Remark \ref{r:holofunctions}). Using Soundararajan's method \cite{Soundararajan} for bounding moments of $L$-functions we will prove the following result.
\begin{theorem} \label{thm:momentbd}
Assume GRH for each of the $L$-functions $L(s, \pi \times \AI(\Lambda))$, $L(s, \ad(\pi) \times \AI(\Lambda^2))$, $L(s, \ad(\pi) \times \chi_d)$, $L(s, \std(\pi) \times \AI(\Lambda^2))$, $L(s, \std(\pi) \times \chi_d)$, and $L(s,\chi_d)$. Then
\[
\frac{1}{|\tmop{Cl}_K|}  \sum_{\Lambda \in \widehat{ \tmop{Cl}_K}} \sqrt{ L(\tfrac12, \pi \times \AI(\Lambda))} \ll_{\pi,\varepsilon} \frac{1}{(\log |d|)^{\frac18-\varepsilon}}.
\]
\end{theorem}
Related moment estimates for families of $L$-functions have been established in \cite{Milinovich-TB, Lester-Radziwill, Huang-Lester, Blomer-Brumley}. Since we assume GRH, we immediately get that $L(\tfrac12, \pi \times \AI(\Lambda)) \ge 0$,  which incidentally also follows from the refined Gan--Gross--Prasad identity using \eqref{e:ggpineq}. However the non-negativity of $L(\tfrac12, \pi \times \AI(\Lambda))$ is known unconditionally due to \cite[Theorem 1.1]{lapid03}, so the square root of the central value is unambiguous, independent of the truth of any unproven hypothesis.


\subsection{Local computation}

For $\tmop{Re}(s)>1$ the $L$-function $L(s,\pi)$ is given by
\[ \tst
L(s, \pi)= \prod_{p} L(s,\pi_p), \qquad L(s,\pi_p)= \left(1-\frac{\alpha_p}{p^s} \right)^{-1}  \left(1-\frac{\alpha_p^{-1}}{p^s} \right)^{-1}  \left(1-\frac{\beta_p}{p^s} \right)^{-1}  \left(1-\frac{\beta_p^{-1}}{p^s} \right)^{-1}
\]
 for $p\nmid N$, where $\mathcal S_p(\pi)=\{\alpha_p, \alpha_p^{-1},\beta_p, \beta_p^{-1}\}$ are the Satake parameters of $\pi$
and it is known that $|\alpha_p|= |\beta_p| =1$ by \cite{weissram}.
Also, given $\Lambda \in \widehat{\tmop{Cl}_K}$ we have for $\tmop{Re}(s)>1$
\[ \tst
L(s, \mathcal{AI}(\Lambda))= \prod_{p}  L(s, \mathcal{AI}(\Lambda_p)), \quad
L(s, \mathcal{AI}(\Lambda_p))=\left(1-\frac{\alpha_{\Lambda_p}}{p^s} \right)^{-1} \left(1-\frac{\beta_{\Lambda_p}}{p^s} \right)^{-1}
\]
 for $p \nmid N$, where the Satake parameters $\mathcal S_p(\AI(\Lambda))=\{\alpha_{\Lambda_p}, \beta_{\Lambda_p}\}$ satisfy:
\begin{itemize}
    \item if $\big( \frac{d}{p} \big) =1$ then $p \mathcal O_k=\mathfrak p \overline {\mathfrak p}$ and $\alpha_{\Lambda_p}=\Lambda(\mathfrak p)$, $\beta_{\Lambda_p}=\Lambda(\overline {\mathfrak p})=\overline{\alpha_{\Lambda_p}}$ ;
    \item if $\big( \frac{d}{p} \big) =-1$ then $\alpha_{\Lambda_p}=1$ and $\beta_{\Lambda_p}=-1$;
    \item if $p|d$ then $p \mathcal O_k=\mathfrak p^2$ and  $\alpha_{\Lambda_{p}}=\Lambda(\mathfrak p)$, $\beta_{\Lambda_p}=0$.
\end{itemize}

In the notation above, we have that the
Satake parameters of $L(s, \tmop{std}(\pi))$ and $L(s, \ad(\pi))$
are respectively given by:

\begin{itemize}



\item $\mathcal S_{p}(\tmop{std}(\pi))=\{ 1, \alpha_p\beta_p, \alpha_p \beta_p^{-1}, \alpha_p^{-1} \beta_p, \alpha_p^{-1} \beta_p^{-1}\}$;



\medskip

\item  $\mathcal S_{p}(\ad(\pi)) =
 \{
\alpha_p^2,\alpha_p^{-2}, \beta_p^2, \beta_p^{-2},
 \alpha_p \beta_p,
\alpha_p^{-1} \beta_p,
 \alpha_p\beta_p^{-1},
  \alpha_p^{-1} \beta_p^{-1},1,1\}$.

\end{itemize}
Define for $\star \in \{ \pi, \AI(\Lambda), \tmop{std}(\pi), \tmop{ad}(\pi) \} $ and $p \nmid N$
\[
a_{\star}(p^n)=\sum_{\alpha \in \mathcal S_{p}(\star)} \alpha^n.
\]
Also, for $f,g \in \{ \pi, \AI(\Lambda), \tmop{std}(\pi), \tmop{ad}(\pi) \}$ let $a_{f \times g}(p^n)=a_{f}(p^n)a_{g}(p^n)$ and $a_{f \times \chi_d}(p^n)=a_{f}(p^n) \chi_d(p^n)$.

We now make the following observations. First, using that $|\alpha_p|=|\beta_p|=1$ it immediately follows that for any $n \ge 1$ and $p \nmid N$
\begin{equation} \label{eq:ramanujan}
|a_{\pi \times \AI(\Lambda)}(p^n) | \le 8.
\end{equation}
Additionally, observe for $p\nmid N$ that
\[
\begin{split}
a_{\tmop{ad}(\pi)}(p)&= \alpha_p^2+\alpha_p^{-2}+\beta_p^{2}+\beta_p^{-2}+a_{\tmop{std}(\pi)}(p)+1.
\end{split}
\]
Hence, for $p \nmid N$
\begin{equation} \label{eq:Fsquare}
    a_{\pi}(p^2)= a_{\tmop{ad}(\pi)}(p) -a_{\tmop{std}(\pi)}(p)-1.
\end{equation}
Also, for $p \nmid N$
\begin{equation} \label{eq:psquare}
\begin{split}
    a_{\pi}(p)^2=&\alpha_p^2+\alpha_p^{-2}+\beta_p^{2}+\beta_p^{-2}+2(\alpha_p+\alpha_p^{-1})(\beta_p+\beta_p^{-1})+4 \\
    =&
    a_{\ad(\pi)}(p)+a_{\std(\pi)}(p)+1.
    \end{split}
\end{equation}
Recall that for $\big(\frac{d}{p} \big)\neq-1$, $\alpha_{\Lambda_p}^2=\Lambda^2(\mathfrak p)=\alpha_{\Lambda_p^2}$. Thus, we get that for $p \nmid N$
\[ \tst
a_{\AI(\Lambda)}(p^2)
= a_{\AI(\Lambda^2)}(p)+\left(\frac{d}{p} \right)^2-\left(\frac{d}{p} \right) = \begin{cases}
 \alpha_{\Lambda_p^2}+\beta_{\Lambda_p^2} & \text{ if } \big( \frac{d}{p} \big) \neq -1,\\
2  & \text{ if } \big( \frac{d}{p} \big) =-1.
\end{cases}
\]
Therefore, combining the above equation with \eqref{eq:Fsquare}
we conclude that
\begin{equation} \label{eq:RSsquare}\tst
a_{\pi \times \AI(\Lambda)}(p^2)=(a_{\tmop{ad}(\pi)}(p) -a_{\tmop{std}(\pi)}(p)-1) \left(a_{\AI(\Lambda^2)}(p)+\left( \frac{d}{p}\right)^2-\left(\frac{d}{p} \right) \right).
\end{equation}

\subsection{Preliminary lemmas} Soundararajan's method for moments of $L$-functions starts with a remarkable bound for the central $L$-value given in terms of a Dirichlet polynomial supported on prime powers. This inequality has been generalized by Chandee \cite{Chandee}.

\begin{lemma} \label{lem:chandee}
Assume GRH for $L(s,\pi \times \AI(\Lambda))$. There exists $C_0>1$, which depends at most on $N,k$ such that for $x \ge 2$
\[
\log L(\tfrac12, \pi \times \AI(\Lambda)) \le \sum_{\substack{p^n \le x \\ p \nmid N}} \frac{a_{\pi \times \AI(\Lambda)}(p^n)}{n p^{\frac{n}{2}(1+1/\log x)}} + \frac{C_0 \log |d|}{\log x}.
\]
\end{lemma}
\begin{remark} \label{rem:upperbd}
Taking $x= \log |d|$ and using \eqref{eq:ramanujan} it follows that
\[ \tst
L(\tfrac12, \pi \times \AI(\Lambda))  \ll \exp\left(2C_0 \frac{\log |d|}{\log \log |d|} \right).
\]
\end{remark}

\begin{proof}
This follows immediately from \cite[Theorem 2.1]{Chandee}.
\end{proof}

To analyze the above sum over prime powers, we will separately consider the primes, squares of primes, and higher prime powers. The contribution of the squares of primes is estimated in the following result.

\begin{lemma} \label{lem:squares}
Assume GRH. Then for $x \ge 2$ each of the following hold:
\begin{itemize}
    \item[i)] $\displaystyle \sum_{\substack{p \le \sqrt{x} \\ p \nmid N}} \frac{a_{\pi \times \AI(\Lambda)}(p^2)}{p^{1+2/\log x}} \frac{\log x/p}{\log x}=-\log \log x(1+o(1))+O_{\pi} \left( \log \log \log |d|\right);$
    \item[ii)] $\displaystyle \sum_{\substack{p \le x \\ p\nmid N, \left( \frac{d}{p}\right)=1 }} \frac{a_{\pi}(p)^2}{p}= \frac12 \log \log x(1+o(1))+O_{\pi}(\log \log \log |d|).$
\end{itemize}
\end{lemma}
\begin{remark}
We assume GRH holds for $L(s,\chi_d),L(s, \pi \times \AI(\Lambda))$, $L(s, \ad(\pi) \times \AI(\Lambda^2))$, $L(s, \ad(\pi) \times \chi_d)$, $L(s, \std(\pi) \times \AI(\Lambda^2))$, and $L(s, \std(\pi) \times \chi_d)$.
\end{remark}
\begin{proof}

Using that $|\alpha_p|=|\beta_p|=1$, we get for $x \ge (\log |d|)^3$
\[ \tst
\sum_{\substack{p \le x \\ p \nmid N}} \frac{a_{\tmop{std}(\pi) \times \chi_d}(p)}{p}=\sum_{\substack{ (\log |d|)^3 \le p \le x  \\ p \nmid N}} \frac{a_{\tmop{std}(\pi) \times \chi_d}(p)}{p}+O(\log \log \log |d|).
\]
By Proposition \ref{p:lfunctionsfacts}, $L(s, \tmop{std}(\pi) \times \chi_d)$ belongs to the extended Selberg class, so that we may use \cite[Lemma 5]{Blomer-Brumley}. Applying this lemma twice it follows that the sum on the r.h.s. is $O_{\pi}(1)$ and so
\[ \tst
\sum_{\substack{p \le x \\ p \nmid N}} \frac{a_{\tmop{std}(\pi)}(p) \left(\frac{d}{p}\right)}{p} \ll_{\pi} \log \log \log |d|.
\]
Repeating the argument above and recalling Remark \ref{r:holofunctions} it follows that for $\star \in \{ \tmop{ad}(\pi) \times \chi_d, \tmop{ad}(\pi) \times \AI( \Lambda^2), \tmop{std}(\pi)\times \AI(\Lambda^2), \AI(\Lambda^2)\}$ that
$
\sum_{\substack{p \le x \\ p\nmid N}} \frac{a_{\star}(p)}{p} \ll\log \log \log |d|.
$
Also, $\sum_{p|d} \frac{1}{p} \ll \log \log \log |d|$ and under GRH we have $\sum_{p \le x} \frac{\chi_d(p)}{p} \ll \log \log \log |d|$. Finally, using that  $L(s,\tmop{std}(\pi))$ and $L(s,\tmop{ad}(\pi))$ do not vanish on the line $\tmop{Re}(s)=1$, the latter of which follows from Proposition \ref{p:lfunctionsfacts},  we have that (e.g. see \cite[Theorem 1]{KacPer03})
\begin{equation} \notag \tst
\sum_{\substack{p \le x \\ p\nmid N}} \frac{a_{\tmop{adj}(\pi)}(p)}{p} =o(\log \log x), \qquad \sum_{\substack{p \le x \\ p\nmid N}} \frac{a_{\std(\pi)}(p)}{p} =o(\log \log x).
\end{equation}

Therefore,
applying \eqref{eq:RSsquare} and using the estimates above along with partial summation to handle the smooth weight $p^{-2/\log x} \frac{\log x/p}{\log x}$ completes the proof of $i)$.

To prove $ii)$ we rewrite the condition $\big( \frac{d}{p} \big)=1$ as $\frac{1}{2}\big(\big(\frac{d}{p}\big)+1\big)$ for $p \nmid d$, use \eqref{eq:psquare} and argue as above.
We trivially bound the contribution of $p|d$ by
$
\le 16 \sum_{p|d}\frac{1}{p} \ll \log \log \log |d|.
$
\end{proof}

\subsection{Large deviations of Dirichlet polynomials} To bound the frequency which large values of $\log L(\tfrac12, \pi \times \AI(\Lambda))$ occur we will estimate large deviations of the sum over primes in Lemma \ref{lem:chandee} by bounding its moments. This is done in the following fairly standard lemma (e.g. \cite[Lemma 3]{Soundararajan}) whose proof we include for completeness.

\begin{lemma} \label{lem:momentbds}
Let $\{b_p\}_p \subset \mathbb R$ and $\ell \in \mathbb N$. Suppose $x^{\ell} < \sqrt{|d|}/2$. Then each of the following hold:
\begin{itemize}
    \item[i)]$ \displaystyle
\frac{1}{|\tmop{Cl}_K|} \sum_{\Lambda \in \widehat{\tmop{Cl}_K}} \bigg(  \sum_{\substack{p \le x \\ p \nmid d, \, p \nmid N}} \frac{b_p a_{\AI(\Lambda)}(p)}{\sqrt{p}}\bigg)^{2\ell} \le \frac{(2\ell)!}{2^{\ell} \ell!} \bigg( 2 \sum_{\substack{p \le x \\ \left( \frac{d}{p}\right) =1, \, p \nmid N}} \frac{b_p^2}{p}\bigg)^{\ell};
$
    \item[ii)] $\displaystyle \frac{1}{|\tmop{Cl}_K|} \sum_{\Lambda \in \widehat{\tmop{Cl}_K}} \bigg(  \sum_{\substack{p \le x \\ p | d, \, p \nmid N}} \frac{b_p a_{\AI(\Lambda)}(p)}{\sqrt{p}}\bigg)^{2\ell} \le \frac{(2\ell)!}{2^{\ell} \ell!}  \bigg( \sum_{\substack{p \le x \\ p|d, \, p \nmid N}} \frac{b_p^2}{p}\bigg)^{\ell}.$
\end{itemize}

\end{lemma}

\begin{proof}
Let $b_n = \prod_{p^j || n} b_p^j$.
We have that
\begin{equation} \label{eq:expand2}
\frac{1}{|\tmop{Cl}_K|} \sum_{\Lambda \in \widehat{\tmop{Cl}_K}} \bigg(  \sum_{\substack{p \le x \\ p \nmid d, p\nmid N}} \frac{b_p a_{\AI(\Lambda)}(p)}{\sqrt{p}}\bigg)^{2\ell}=
\sum_{\substack{n \ge 1  \\ (n,N)=1}} \frac{b_n}{\sqrt{n}} \sum_{\substack{ \mathfrak p_1 , \ldots, \mathfrak p_{2\ell} \subset \mathcal O_K \\ N(\mathfrak p_1 \cdots \mathfrak p_{2\ell})=n \\ N(\mathfrak p_j) \le x \\ \mathfrak p_j \text{ split }}} \frac{1}{|\tmop{Cl}_K|} \sum_{\Lambda \in \widehat{\tmop{Cl}_K}}
\prod_{j=1}^{2\ell} (\Lambda(\mathfrak p_j)+\Lambda(\mathfrak p_j)^{-1} ).
\end{equation}
For $n \in \mathbb N$ in the sum above write $n=q_1^{e_1} \cdots q_r^{e_r}$ where $q_1, \ldots, q_r$ are distinct primes with $q_j \mathcal O_K=\mathfrak q_j \overline{\mathfrak q_j}$ for each $j=1,\ldots, r$. The inner sum on the r.h.s. above equals
\[\tst
  \sum_{\substack{0 \le f_1\le e_1 }} \cdots \sum_{0 \le f_r \le e_r}\binom{e_1}{f_1} \cdots \binom{e_r}{f_r}  \frac{1}{|\tmop{Cl}_K|} \sum_{\Lambda \in \widehat{\tmop{Cl}_K}} \Lambda\left(  \mathfrak q_1^{f_1} \overline{ \mathfrak q_1}^{e_1-f_1} \cdots  \mathfrak q_r^{f_r} \overline{ \mathfrak q_r}^{e_r-f_r}  \right).
\]
Write $\mathfrak a= \mathfrak q_1^{f_1} \overline{ \mathfrak q_1}^{e_1-f_1} \cdots  \mathfrak q_r^{f_r} \overline{ \mathfrak q_r}^{e_r-f_r}  $. Since $N(\mathfrak a) \le x^{2\ell} < |d|/4$, the inner sum above is zero unless $\mathfrak a=(\alpha)$ and $\alpha \in \mathbb Z$. Hence, the only terms that do not vanish are those with $e_1, \ldots, e_r$ even and $f_j=e_j/2$ for $j=1,\ldots, r$. Define the multiplicative functions $\nu(n)$ by $\nu(p^a)=1/a!$ and $g(n)$ by $g(p^a)=\binom{a}{a/2}$ if $2|a$ and $g(p^a)=0$ otherwise. The r.h.s. of \eqref{eq:expand2} equals
\[
(2\ell!) \sum_{\substack{p|n \Rightarrow  p \le x, \,
p\nmid N \, \& \, \left(\frac{d}{p}\right)=1  \\ \Omega(n)=2\ell}} \frac{b_n \nu(n) g(n)}{\sqrt{n}}  \le \frac{(2\ell!)}{2^{\ell} \ell!} \bigg( 2 \sum_{\substack{p \le x \\ \left(\frac{d}{p}\right)=1, p \nmid N }} \frac{b_p^2}{p}\bigg)^{\ell},
\]
where for the last inequality we used that the sum on the l.h.s. above is supported on squares and $ \tst \nu(m^2)g(m^2)\le \nu(m).$
This completes the proof of $i)$.

Similarly, we have that
\[
\frac{1}{|\tmop{Cl}_K|} \sum_{\Lambda \in \widehat{\tmop{Cl}_K}} \bigg(  \sum_{\substack{p \le x \\ p | d, p \nmid N}} \frac{b_p a_{\AI(\Lambda)}(p)}{\sqrt{p}}\bigg)^{2\ell}=
\sum_{n \ge 1} \frac{b_n}{\sqrt{n}} \sum_{\substack{ \mathfrak p_1 , \ldots, \mathfrak p_{2\ell} \subset \mathcal O_K \\ N(\mathfrak p_1 \cdots \mathfrak p_{2\ell})=n \\ N(\mathfrak p_j) \le x \\ \mathfrak p_j \text{ ramified }}} \frac{1}{|\tmop{Cl}_K|} \sum_{\Lambda \in \widehat{\tmop{Cl}_K}}
 \Lambda(\mathfrak p_1 \cdots \mathfrak p_{2\ell}).
\]
Write $\mathfrak a=\mathfrak p_1 \cdots \mathfrak p_{2\ell}$. The inner sum is zero unless $\mathfrak a=(\alpha)$ with $\alpha \in \mathbb Z$, since $N(\mathfrak a ) < |d|/4$. Consequently, $n$ is a square. The r.h.s. above equals
\[
(2\ell)!\sum_{\substack{p|n \Rightarrow p  \le x, \, p \nmid N \, \&  \, p |d  \\ \Omega(n)=\ell}} \frac{b_n^2}{n} \nu(n^2) \le \frac{(2\ell)!}{2^{\ell} \ell!}   \bigg( \sum_{\substack{p \le x \\ p|d,\, p \nmid N}} \frac{b_p^2}{p}\bigg)^{\ell},
\]
where we used that $\nu(n^2) \le \frac{\nu(n)}{2^{\Omega(n)}}$ in the last step.
This completes the proof of $ii)$.
\end{proof}

For $x \ge 2$, let
$
P(\Lambda;x)= \sum_{\substack{p \le x \\ p \nmid N}} \frac{a_{\pi}(p) a_{\AI(\Lambda)}(p)}{p^{\frac12+\frac{1}{\log x}}} \frac{\log x/p}{\log x}
$
and for $V\in \mathbb R$ let
$
A_{K}(V;x)= \frac{1}{|\tmop{Cl}_K|} | \{ \Lambda \in \widehat{\tmop{Cl}}_K : P(\Lambda;x) > V \}|.
$

\begin{lemma} \label{lem:largedev}
Let $\epsilon>0$ be sufficiently small. Then
for $(\log \log |d|)^{3/4}\le V \le  \frac{ 4C_0 \log |d|}{\log \log |d|}$ we have that
\[
A_{K}(V;|d|^{\frac{1}{\epsilon V}}) \ll_{\pi, \epsilon} \exp\bigg(\frac{-V^2}{2 \log \log |d|}(1-2\epsilon)\bigg)+e^{-\frac{\epsilon}{4} V \log V}.
\]
\end{lemma}
\begin{proof}
Set $x=|d|^{\frac{1}{\epsilon V}}$ and $z=x^{1/\log \log |d|}$.
Let
\[
P(\Lambda;x)=P_1(\Lambda)+P_2(\Lambda)+P_3(\Lambda)
\]
where
\[
P_1(\Lambda) =\sum_{\substack{p \le z \\ p \nmid d, p \nmid N}} \frac{a_{\pi}(p) a_{\AI(\Lambda)}(p)}{p^{\frac12+\frac{1}{\log x}}} \frac{\log x/p}{\log x}, \quad P_2(\Lambda) =\sum_{\substack{ z <p \le x \\ p \nmid d, \, p \nmid N}} \frac{a_{\pi}(p) a_{\AI(\Lambda)}(p)}{p^{\frac12+\frac{1}{\log x}}} \frac{\log x/p}{\log x},
\]
and $P_3(\Lambda)=P(\Lambda;x)-P_1(\Lambda)-P_2(\Lambda)$. Let $V_1=(1-\epsilon)V$, $V_2=\tfrac{\epsilon}{2}V$.
Clearly, if $P(\Lambda;x)>V$ then
$
i) \, P_1(\Lambda)>V_1, \, ii) \, P_2(\Lambda)>V_2, \text{ or }  iii) \, P_3(\Lambda)>V_2.
$
By Lemma \ref{lem:momentbds} we get that the number of $\Lambda \in \widehat{\tmop{Cl}_K}$ such that $i)$ holds is bounded by
\[
\begin{split}
\frac{1}{V_1^{2\ell}}  \sum_{\Lambda \in \widehat{\tmop{Cl}}_K} P_1(\Lambda)^{2\ell} \ll |\tmop{Cl}_K| \frac{(2\ell)!}{(V_1^22)^{\ell} \ell!} \bigg( 2 \sum_{\substack{p \le d \\ \left(\frac{d}{p} \right)=1, \, p \nmid N}} \frac{a_{\pi}(p)^2}{p}  \bigg)^{\ell}
\end{split}
\]
provided that $\ell \le \frac13 \epsilon V \log \log |d|$.
Applying Lemma \ref{lem:squares} and Stirling's formula the r.h.s. above is
$
\ll_{\pi,\epsilon} |\tmop{Cl}_K| \left(\frac{2\ell \log \log |d| (1+\epsilon^3)}{V_1^2 e} \right)^{\ell}.
$
For $V \le \frac13 \epsilon (\log \log |d|)^2$ set $\ell=\lfloor \frac{V_1^2}{2\log \log |d|} \rfloor$ and for larger $V$ take $\ell= \lfloor \epsilon V/3 \rfloor$. Hence,
\begin{equation} \label{eq:smallbd}
\frac{1}{|\tmop{Cl}_K|} |\{ \Lambda \in \widehat{\tmop{Cl}_K}: P_1(\Lambda)> V_1 \}| \ll_{\pi, \epsilon} \exp\bigg(\frac{-V^2(1-2\epsilon)}{2 \log \log |d|}\bigg)+e^{-\frac{\epsilon}{4} V \log V}.
\end{equation}

Arguing as before, taking $\ell =\lfloor \epsilon V/3\rfloor$ and using the bound $|a_{\pi}(p)|^2 \le 16$ we have that
\begin{equation} \label{eq:largebd}
\begin{split}
&\tst \frac{1}{|\tmop{Cl}_K|} |\{ \Lambda \in \widehat{\tmop{Cl}}_K: P_2(\Lambda)> V_2 \}| \ll   \frac{(2\ell)!}{(V_2^22)^{\ell} \ell!} \bigg( 2 \sum_{\substack{z \le p \le x \\ p \nmid N }} \frac{a_{\pi}(p)^2}{p}  \bigg)^{\ell} \\
 &\qquad  \qquad\ll  \tst  \left( \frac{65 \ell  \log \log \log |d|}{V_2^2 }  \right)^{\ell}
 \ll_{\epsilon}   e^{-\frac{\epsilon}{4}V \log V}.
\end{split}
\end{equation}
Similarly,
\begin{equation} \label{eq:dbound}
\begin{split}
\frac{1}{|\tmop{Cl}_K|} |\{ \Lambda \in \widehat{\tmop{Cl}_K}: P_3(\Lambda)> V_2 \}| \ll  &  \frac{(2\ell)!}{(V_2^22)^{\ell} \ell!} \bigg(  \sum_{\substack{p|d, \, p \nmid N}} \frac{a_{\pi}(p)^2}{p}  \bigg)^{\ell}
 \ll_{\epsilon}   e^{-\frac{\epsilon}{4}V \log V}.
\end{split}
\end{equation}
Combining \eqref{eq:smallbd}, \eqref{eq:largebd}, and \eqref{eq:dbound} completes the proof.
\end{proof}
\subsection{Proof of Theorem \ref{thm:momentbd}} Let
$
B_K(V)=\frac{1}{|\tmop{Cl}_K|} | \{ \Lambda \in \widehat{\tmop{Cl}}_K : \log L(\tfrac12, \pi \times \AI(\Lambda)) > V \}|.
$
Observe that
\begin{equation}\label{eq:IBP}
\begin{split} \tst
\frac{1}{|\tmop{Cl}_K|} \sum_{\Lambda \in \widehat{\tmop{Cl}_K} } \sqrt{L(\tfrac12, \pi \times \AI(\Lambda))}&=\frac12 \int_{\mathbb R} e^{V/2} B_K(V) \, dV\\
&= \tst\frac{1}{2 (\log |d|)^{1/4}} \int_{\mathbb R} e^{V/2} B_K(V-\tfrac12 \log \log |d|) \, dV.
\end{split}
\end{equation}
It suffices to consider $2(\log \log |d|)^{3/4} \le V \le 2 C_0 \frac{\log |d|}{\log \log |d|}$ as the contribution of $V \le 2(\log \log |d|)^{3/4}$ is trivially  $O_{\eps}((\log |d|)^{-1/4+\varepsilon})$ and by Remark \ref{rem:upperbd} we have $B_K(V)=0$ for $V \ge  2 C_0 \frac{\log |d|}{\log \log |d|}$.

By Lemmas \ref{lem:chandee}, \ref{lem:squares}, and \eqref{eq:ramanujan} we have for $x \le |d|$ that
\begin{equation} \label{eq:ineq1}\tst
\log L(\tfrac12, \pi \times \AI(\Lambda)) \le P(\Lambda;x)-\tfrac12 \log \log |d|+2C_0\frac{ \log |d|}{\log x}+o( \log \log |d|).
\end{equation}
Taking $x=|d|^{\frac{1}{\epsilon V}}$
it follows from \eqref{eq:ineq1} that
$
B_K(V-\tfrac12 \log \log |d|) \le A_K(V(1-3C_0\epsilon); |d|^{\frac{1}{\epsilon V}}).
$
Applying the above estimate and using Lemma \ref{lem:largedev} along with the identity
\[
\int_{\mathbb R} e^{\frac{-t^2}{2 \log \log |d|}+t/2} \, dt = \sqrt{2\pi \log \log |d|} (\log |d|)^{\frac18}
\]
 we get that the r.h.s. of \eqref{eq:IBP} is
\[
\tst  \ll  (\log |d|)^{-\frac14}  \int_{2(\log \log |d|)^{3/4}}^{2C_0 \frac{\log |d|}{\log \log |d|}} \bigg( e^{\frac{-V^2(1-7C_0 \varepsilon)}{2\log \log |d|}} + e^{-\frac{\epsilon}{5}V \log V} \bigg) \, dV+ (\log |d|)^{-\frac14+\varepsilon}
\ll_{\pi,\eps}  \frac{1}{(\log |d|)^{\frac18-\varepsilon}},
\]
which completes the proof. \qed

\bibliography{fundamental}

\end{document}